\documentclass[a4paper]{amsart}

\pdfoutput=1

\usepackage[utf8]{inputenc}
\usepackage[T1]{fontenc}
\usepackage{lmodern}
\usepackage{amsthm, amssymb, amsmath, amsfonts, mathrsfs, dsfont, esint}
\usepackage[colorlinks=true, pdfstartview=FitV, linkcolor=blue, citecolor=blue, urlcolor=blue,pagebackref=false]{hyperref}
\usepackage{microtype}

\newtheorem{thm}{Theorem}[section]

\newtheorem{lem}[thm]{Lemma}
\newtheorem{cor}[thm]{Corollary}
\theoremstyle{remark}

\theoremstyle{definition}

\renewcommand{\le}{\leqslant}
\renewcommand{\ge}{\geqslant}
\renewcommand{\leq}{\leqslant}
\renewcommand{\geq}{\geqslant}
\newcommand{\ls}{\lesssim}
\renewcommand{\subset}{\subseteq}

\newcommand{\B}{\mathbb{B}}

\newcommand{\N}{\mathbb{N}}

\newcommand{\Ll}{\left}
\newcommand{\Rr}{\right}

\newcommand{\1}{\mathbf{1}}
\newcommand{\R}{\mathbb{R}}

\newcommand{\Z}{\mathbb{Z}}
\newcommand{\X}{\mathcal{X}}
\newcommand{\Y}{\mathcal{Y}}

\renewcommand{\P}{\mathbb{P}}

\newcommand{\td}{\tilde}
\newcommand{\eps}{\varepsilon}

\numberwithin{equation}{section}

\title[Quantitative homogenization of degenerate random environments]{Quantitative homogenization of degenerate random environments}
\author{Arianna Giunti, Jean-Christophe Mourrat}

\address[Arianna Giunti]{Max Planck Institute for Mathematics in the Sciences, Leipzig, Germany}

\address[Jean-Christophe Mourrat]{Ecole normale sup\'erieure de Lyon, CNRS, Lyon, France}

\begin{document}

\begin{abstract}

We study discrete linear divergence-form operators with random coefficients, also known as random conductance models. We assume that the conductances are bounded, independent and stationary; the law of a conductance may depend on the orientation of the associated edge. We give a simple necessary and sufficient condition for the relaxation of the environment seen by the particle to be diffusive, in the sense of every polynomial moment. As a consequence, we derive polynomial moment estimates on the corrector.

\bigskip



\medskip


\end{abstract}
\maketitle
%
%
%
%
%
%
%
%
\section{Introduction}
\label{s.intro}

We study the homogenization of discrete divergence-form operators
\begin{equation}
\label{e.gen}
-\nabla \cdot a \nabla f(x) := \sum_{y \sim x} a((x,y)) (f(y) - f(x)),
\end{equation}
where $a = (a(e))_{e \in \B}$ is a family of independent random variables indexed by the nearest-neighbor, unoriented edges of the graph $\Z^d$, { $d > 2$}. We assume that the coefficients $a(e)$ take values in $[0,1]$, that the law of $a(e)$ depends only on the orientation of the edge $e$, and that none of these $d$ probability laws is a Dirac mass at~$0$. In this setting, we give a necessary and sufficient condition for the relaxation of the ``environment viewed by the particle'' to be diffusive in the sense of every polynomial moment. As the name implies, this process can be described in terms of the random walk with generator given by~\eqref{e.gen}. We will rather define the flow of its semigroup directly by means of the PDE \eqref{E1} below. Denoting by $u_t$ the solution to \eqref{E1} with bounded, local and centered initial condition $g$, we show that the property
\begin{equation}
\label{e.c1}
\mbox{for every  $p \ge 1$,} \quad \sup_{t \ge 1} \, t^{\frac d 2} \, \langle |u_t|^{2p} \rangle^{\frac 1 p} < \infty
\end{equation}
holds as soon as
\begin{equation}
\label{a.mom.bounds}
\mbox{for every $q \ge 1$,} \quad \langle \bigl(\sup_{1 \le i \le d} a(e_i)\bigr)^{-q}	\rangle < \infty,
\end{equation}
where $(e_1,\ldots, e_d)$ is the canonical basis of $\Z^d$, and $\langle\, \cdot \, \rangle$ denotes the expectation with respect to the law of $\{a(e)\}$. Moreover, if \eqref{e.c1} holds for a single ``generic'' bounded and local initial datum $g$, then \eqref{a.mom.bounds} holds as well. Finally, we show that if the initial condition $g$ in \eqref{E1} is the divergence of a local bounded function and if \eqref{a.mom.bounds} holds, then \eqref{e.c1} can be improved to
\begin{equation}
\label{e.c1+}
\mbox{for every $p \ge 1$ and $\eps > 0$,} \quad  \sup_{t \ge 1} \, t^{\frac d 2+1-\eps} \,\langle |u_t|^{2p} \rangle^{\frac 1 p} < \infty.
\end{equation}
A rather degenerate example of environment satisfying \eqref{a.mom.bounds} can be constructed by letting $a(e)$ be i.i.d.\ Bernoulli random variables in $(d-1)$ directions, and letting $a(e) = (1+E_e)^{-1}$ with $(E_e)$ i.i.d.\ exponential random variables in the remaining direction.

As shown in \cite{M,GNO} and recalled below, these estimates imply a range of other quantitative homogenization results, including bounds on the corrector. They can also be used to prove quenched central limit theorems for the associated random walk. 

Under the condition that the coefficients $a(e)$ are uniformly bounded away from~$0$ and infinity, the estimate \eqref{e.c1} and the stronger estimate \eqref{e.c1+} with $\eps = 0$ were proved in \cite{GNO}. We refer to \cite{AS, GNO3, GO5, AKM2} for more recent developments under this assumption. Bounds on the corrector under assumptions similar to ours and for $d \ge 3$ were obtained in \cite{lno}. The case of coefficients that are bounded away from $0$ but not from infinity was considered in \cite{M,bumo}.

Our approach is inspired by the strategy of \cite{GNO}, which rests on quenched heat kernel bounds. In the context of degenerate environments, quenched diffusive bounds on the heat kernel are false in general. However, under the condition \eqref{a.mom.bounds}, we will be able to control the anomalous behavior of the heat kernel, in the sense of every polynomial moment, by exploiting the method presented in \cite{MO}. We then show that these weaker bounds are sufficient to imply \eqref{e.c1} and \eqref{e.c1+}.

\smallskip

Our long-term goal is to develop a comparable strategy in the context of interacting particle systems, in particular to study the relaxation of the environment viewed by a tagged particle in the symmetric exclusion process. We expect the results of this paper to be a first step in this direction. Loosely speaking, in the present work, we leverage on the existence of one ``good direction'' where conductances are well-behaved. For the exclusion process, we hope to benefit from the good behavior of the model in the \emph{time} direction, as illustrated for instance by \cite[Lemma~5.3]{MO}.

The results we present here shed light on the associated process of the random walk among random conductances. This is the Markov process with infinitesimal generator given by \eqref{e.gen}. In this view, the corrector provides us with harmonic coordinates that turn the walk into a martingale. As is well-known, these coordinates allow to show an \emph{annealed} invariance principle for the random walk, under very general conditions on the conductances \cite{kipvar,dfgw}. The qualifier ``annealed'' indicates that convergence in law is only known if one averages over the environment as well as on the trajectories of the walk. When the conductances are uniformly elliptic, it was quickly realized \cite{osada} that the statement can be improved to a \emph{quenched} invariance principle: that is to say, one that holds for almost every realization of the environment. What needs to be shown is that the corrector, evaluated at the position of the random walk, is of lower order compared with the position of the walk itself, with probability one with respect to the environment. By general arguments, one only knows that the corrector is sublinear in an $L^2$-averaged sense, and this is not sufficient in itself to guarantee a quenched result. One possibility to overcome this difficulty is to show that the walk is sufficiently ``spread out'' (in the sense that it satisfies heat kernel estimates), so that the averaged information on the corrector becomes sufficient to conclude. This was the route explored in a majority of papers on the subject \cite{osada, sidszn, matpia, berbis, bispre, mathieu, abdh}; see however \cite{fankom1, barham,ger1,ger4} for approaches more similar to ours. The results we derive here give much more precise information than these earlier works, since Corollary~\ref{c.corrector} below implies that the corrector is not only sublinear, but in fact grows slower than any power of the distance, with probability one.

\subsection{Notation and main result}

We say that $x,y \in \Z^d$ are neighbors, and write $x \sim y$, when $|x-y| = 1$. This endows $\Z^d$ with a graph structure, so that we may introduce the associated set of unoriented edges $\B$. 
Throughout the paper, we will typically denote points of $\Z^d$ by $x,y,z$, and edges in $\B$ by $b,e$. For a given edge $e\in \B$, we write $\underline{e}$ and $\overline e$ to denote its two endpoints, with the convention that $\overline e - \underline e = e_i$ for some $i \in \{1,\ldots, d\}$, where we recall that $(e_1,\ldots, e_d)$ is the canonical basis of $\Z^d$. We identify the vector $e_i \in \Z^d$ with the edge $(0,e_i)$, for each $i \in \{1,\ldots d\}$. 

The space of ``environments'' we consider is $\Omega := [0,1]^{\B}$. The group $\Z^d$ naturally acts by translations on $\Omega$ in the following way: for every $x\in \Z^d$ and $a= (a(e))_{e\in \B} \in \Omega$, we define
\begin{equation}
\label{e.def.tau}
\tau_x a := (a (x+e))_{e\in \B},
\end{equation}
where for $e\in \B$, we write $x+e := (x+\underline e, x+\overline e )$.
We consider a random $a = (a (e))_{e \in \B} \in \Omega$ whose law we denote by $\langle \, \cdot \, \rangle$. We assume 
the family of random variables $a = (a(e))$ to be independent and stationary, i.e. for every $x\in \Z^d$, the
random variables $\tau_xa$ and $a$ have the same law. In other words, the random variables $(a(e))$ are independent, and the law of $a(e)$ only depends on the orientation of the edge $e$. We assume that for every $e$, $\P[a(e) = 0] < 1$, since otherwise the model would truly be defined on a 
lower-dimensional space.

\medskip

For a random variable $\xi: \Omega \rightarrow \R$ and a fixed edge $b\in \B$, we define
\begin{align*}
D_b \xi = D_b\xi(a) := \xi(\tau_{\overline b}a) - \xi(\tau_{\underline b}a),
\end{align*}
and simply write $D \xi $ for the $d$-dimensional random vector defined as
$$
(D\xi)_i:= D_{e_i}\xi.
$$
We observe that for every $p\in [1,+\infty]$ the operator $D: L^p(\Omega) \rightarrow L^p(\Omega)^d$ is bounded and that its adjoint in $L^2(\Omega)$, which we denote by $D^*: L^2(\Omega)^d \rightarrow L^2(\Omega)$, is defined as
$$
D^*\xi:=\sum_{i=1}^d D^*_i\xi_i(a), \qquad \text{where} \quad  D_i^*\xi_i:= \xi(\tau_{-e_i}a) - \xi(a).
$$
Given a random variable $g \in L^1(\Omega)$, with $ \langle g \rangle = 0$, our goal is to understand the relaxation to equilibrium of $u : \R_+ \times \Omega \to \R$, solution of
\begin{equation}\label{E1}
\begin{cases}
& \partial_tu_t+D^*aDu_t=0 \\
& u_0 = g,
\end{cases}
\end{equation} 
where
$$
D^*aDu_t(a) := \sum_{i=1}^d D^*_i(a(e_i)D_iu_t(a))= \sum_{e\ni 0}a(e) D_e u_t(a).
$$
Whenever no confusion occurs, we write $u_t$ instead of $u_t(a)$. For $N \in \N$, we say that a function $g : \Omega \to \R$ is local with support of size $N$ if $g$ depends only on a finite number of conductances $\{ a(e^{(1)}), \ldots, a(e^{(N)})\}$.
Here is our main result.
\begin{thm}
\label{t.main} Under the moment condition \eqref{a.mom.bounds}, the following statements hold.
\begin{itemize}
\item[(a)] For every $p \in [1,\infty)$, there exists a constant  $C=C(d,p)< +\infty$ such that if $g : \Omega \to \R$ is local with support of size $N$, bounded and centered, then 
\begin{align}\label{Teo1}
\Ll\langle |u_t|^{2p} \Rr\rangle^{1/p} \le C N^{2} ||g||^2_{L^\infty(\Omega)} \, t^{-\frac d 2 }.
\end{align}
\item[(b)]For every $p \in [1,\infty)$ and $\eps > 0$, there exists a constant $C=C(d,p,\eps) < \infty$ such that if $f: \Omega \to \R$ is local with support of size $N$ and bounded, and if $g = D^* f$, then
\begin{align}\label{Teo2}
\Ll\langle |u_t|^{2p} \Rr\rangle^{1/p} \leq C N^{ 2}||f||^2_{L^\infty(\Omega)}\,  t^{-\Ll(\frac d 2 + 1-\eps\Rr)}.
\end{align}
\end{itemize}
\end{thm}

\subsection{Consequences of the main result}
As was shown in \cite[Section~9]{M} and \cite[Section~6]{GNO}, Theorem \ref{t.main} implies a host of other results of interest in stochastic homogenization. In particular, estimates on the corrector can be derived, by integration in time, from the relaxation to equilibrium of the solution to \eqref{E1} with $g = D^*(ae)$.

\begin{cor}\label{c.corrector}
{  Assume that} the moment condition \eqref{a.mom.bounds} holds, and let $e \in \mathbb{B}$. If $d > 2$, then there exists $\phi_e \in L^2(\Omega)$ solution to the equation
\begin{align}\label{corrector}
D^*a D\phi_e = -D^*a e .
\end{align} 
Moreover, $\phi_e$ is in $L^p(\Omega)$ for every $p \in [1,\infty)$.
\end{cor}
We refer to \cite[Proposition~4]{GNO} for the proofs of these results.
As another example, we can estimate the corrector $\phi_{e,\mu}$ with massive term $\mu > 0$, i.e. the solution of 
\begin{align*}
\mu \phi_{e,\mu}+ D^*aD\phi_{e,\mu}=-D^*a e \, .
\end{align*}
By \cite[Proposition~5]{GNO}, we obtain that for every $\eps > 0$ and $p \in [1,\infty)$,
$$
\begin{aligned}
 \sup_{\mu > 0} \, & \langle |\phi_{e,\mu}|^p \rangle^{\frac 1 p} < \infty .
\end{aligned}
$$

\subsection{On the necessity of the moment condition}
We now explain why our assumption (\ref{a.mom.bounds}) on the law of $(a(e))$ is necessary in order to have the optimal relaxation 
decay (\ref{e.c1}). In this subsection we introduce the notation $\lesssim$ for $\leq C$ where the constant $C$ only depends on the dimension $d$ of the lattice $\Z^d$.

\medskip

If (\ref{a.mom.bounds}) does not hold, then we can find $p_0 \geq 1$ and a sequence
$\{ \varepsilon_{p_0,n} \}_{n \in \N}$, $\varepsilon_{p_0,n} \downarrow 0$ such that for every $n\in \N$,
\begin{align}\label{NSC1}
\mathbb{P}(\sup_{i=1,...,d}a(e_i) \leq \varepsilon_{p_0,n}) = \Pi_{i=1}^d\mathbb{P}_i(a(e_i) \leq \varepsilon_{p_0,n}) \geq \varepsilon_{p_0,n}^{p_0}.
\end{align}
We show that the solution $u_t$ of
\begin{equation}\label{NSC4}
\begin{cases}
& \partial_tu_t+D^*aDu_t=0 \\
& u_0 = g,
\end{cases}
\end{equation}
with $g= a(\tilde e) - \langle a(\tilde e) \rangle$ for a fixed $\tilde e$ such that $0\notin \tilde e$, does not satisfy the bound~\eqref{e.c1}. We make this choice of initial datum $g$ for convenience, but as will be seen shortly, this is inessential. From (\ref{E1}), we may bound by stationarity
\begin{align*}
|\partial_t u_t| \lesssim \bigl (\sup_{e\ni 0}a(e)\bigr)\  ||u_t||_{L^\infty} =  \bigl(\sup_{i=1,...,d}a( \pm e_i)\bigr) \ ||u_t||_{L^\infty},
\end{align*}
and hence, by the maximum principle,
\begin{align*}
|\partial_t u_t| \lesssim \bigl(\sup_{i=1,...,d}a(\pm e_i)\bigr) \ ||g||_{L^\infty}.
\end{align*}
Therefore, if $ \bigl(\sup_{i=1,...,d} a(\pm e_i) \bigr) \leq \frac{1}{t^2}$, we get
\begin{align}\label{NSC3}
|u_t| \geq |g| - \frac{||g||_{L^\infty(\Omega)}}{t},
\end{align}
and thus
\begin{align}
\langle |u_t|^{2q} \rangle &\geq \langle |u_t|^{2q} \1_{\sup_i a(\pm e_i) \leq \frac{1}{t^2}}\rangle = \langle \ |u_t|^{2q} \ \bigl| \  \sup_i a(\pm e_i) \leq \frac{1}{t^2} \ \rangle \, \mathbb{P}( \sup_i a(\pm e_i) \leq \frac{1}{t^2} )\notag\\
&\stackrel{(\ref{NSC3})}{\gtrsim} \langle \bigl(|g| - \frac{||g||_{L^\infty(\Omega)}}{t} \bigr)^{2q} \ \bigl |  \  \sup_i a(\pm e_i) \leq \frac{1}{t^2} \ \rangle \, \mathbb{P}( \sup_i a(\pm e_i) \leq \frac{1}{t^2} )\notag\\
&\stackrel{a(\tilde e) \perp a(\pm e_i)}{=}\langle \bigl(|g| - \frac{||g||_{L^\infty(\Omega)}}{t} \bigr)^{2q} \ \rangle \mathbb{P}( \sup_i a(\pm e_i) \leq \frac{1}{t^2} )\notag\\
&= \langle \bigl(|g| - \frac{||g||_{L^\infty(\Omega)}}{t} \bigr)^{2q} \ \rangle \mathbb{P}( \sup_i a(e_i) \leq \frac{1}{t^2}\, , \, \sup_i a( - e_i) \leq \frac{1}{t^2} )\notag\\
&= \langle \bigl(|g| - \frac{||g||_{L^\infty(\Omega)}}{t} \bigr)^{2q} \ \rangle \bigl(\mathbb{P}( \sup_i a(e_i) \leq \frac{1}{t^2})\bigr)^2,
\end{align}
where in the last equality we use that $\sup_i a(e_i)$ and $\sup_i a(-e_i)$ are independent and have the same law.
For $\{t_n\}_{n \in \N}$ with $t_n^2 = \frac{1}{\varepsilon_{p_0,n}}$ we may apply (\ref{NSC1}) and estimate for every $n \in \N$
\begin{align*}
\langle |u_{t_n}|^{2q} \rangle^{\frac 1 q} \gtrsim \langle \bigl(|g| - \frac{||g||_{L^\infty(\Omega)}} \, {t_n} \bigr)^{2q} \ \rangle^{\frac 1 q} t_n^{-4\frac{p_0}{q}}\gtrsim \biggl(\langle |g|^{2q} \rangle^{\frac 1 q} -\frac{||g||^2_{L^\infty(\Omega)}}{t_n^2}\biggr) t_n^{-4\frac{p_0}{q}}.
\end{align*}
Thus, for any $q > \frac{8p_0}{d}$ we contradict (\ref{e.c1}).
\subsection{Organisation of the paper}
In the rest of the paper, we assume that the moment condition \eqref{a.mom.bounds} holds. We derive the necessary heat kernel bounds in Section~\ref{s.hk}, and proceed to prove Theorem~\ref{t.main} in Section~\ref{s.main}.

\section{Heat kernel bounds}
\label{s.hk}
We say that a random field $\zeta : \Omega \times \Z^d \rightarrow \R$ is stationary if for every $x\in \Z^d$, we have $\zeta(a,x)  =\zeta(\tau_x a, 0)$. Conversely, given a random variable $\xi : \Omega \rightarrow \R$, we define its stationary extension $\bar \xi : \Omega \times \Z^d \rightarrow \R$ as the random field given by $\bar \xi(a, x) := \xi(\tau_x a)$.
If the function $u : \R_+ \times \Omega \to \R$ solves (\ref{E1}), then its stationary extension $\bar u_t(x,a) = u_t(\tau_x a)$ is a solution in $\R_+ \times \Z^d$ of the parabolic PDE
\begin{align}\label{E2}
\begin{cases}
& \partial_t\bar u_t + \nabla^* a \nabla \bar u_t =0 \\
&\bar u_0 = \bar g,
\end{cases}
\end{align}
with $\nabla$ the spacial discrete gradient defined, for an edge $b\in \B$ and a random field $\zeta$, as 
\begin{align*}
\nabla \zeta (a,b) := \zeta(a, \overline b) - \zeta(a, \underline b),\\
a \nabla\zeta (a,b):= a(b)\nabla \zeta(a,b),
\end{align*}
and $\nabla^*$ the adjoint of $\nabla$ in $\ell^2(\Z^d)$.

\medskip

Let $p_t= p_t(a, x, y)$ be the parabolic Green function associated to the operator $\partial_t + \nabla^*a \nabla$, i.e. for every $y\in \Z^d$ the unique bounded solution in $\Z^d$ of
\begin{align}\label{E3}
\begin{cases}
& \partial_t p_t( a, \cdot , y) + \nabla^*a \nabla p_t( a, \cdot , y) =0 \\
&p_0(\cdot ,y) = \1_{\{ y \}}(\cdot ),
\end{cases}
\end{align}
with $\1_{y}$ being the indicator function defined for $x \in \Z^d$
$$
\1_{y}( x) = \begin{cases}
1 \ \ \text{if  $x = y$}\\
0 \ \ \text{ otherwise.}
\end{cases}
$$
For every $\alpha \in \R$, $x \in \R^d$ and $t \ge 0$, we write
\begin{align}\label{LG1}
\omega_\alpha(t, x):=\bigl( \frac{|x|^2}{t+1}+1)^{\frac{\alpha}{2}}.
\end{align}

\medskip

The goal of this section is to show the heat kernel upper bound summarized in the following Lemma~\ref{MO}, which we then lift to an estimate on the gradient of the heat kernel in Lemma~\ref{LG}.

\begin{lem}\label{MO}
Let $p_t(x,y)=p_t(a, x,y)$ be as in (\ref{E3}). There exists a random variable $\X$ such that for all $p\in [1, +\infty)$,
\begin{align}\label{G0}
\langle |\X|^p \rangle \leq C(p) < +\infty ,
\end{align}
and
\begin{align}\label{G1}
p_t(0,0) \le \frac{\X}{t^{d/2}}
\end{align}
or, equivalently
\begin{align}\label{G2}
\sum_{x \in \Z^d} p_{t/ 2}(x,0)^2 \le \frac{\X}{t^{d/2}}.
\end{align}
\end{lem}
\begin{lem}\label{LG} 
For every $ \alpha > \frac{d}{2} +1$, there exists $C=C(\alpha, d) < +\infty$ such that for every $t\in \R^+$ it holds 
\begin{align}
 &\omega_{\alpha-1}(t,x)p_t(x,0) \leq C \frac{\sqrt{\overline{\mathcal{X}}(0)\overline{\mathcal{X}}(x)}}{t^{\frac d 2}},\label{LG2a}\\
\sum_{b\in \mathbb{B}^d}&\omega_\alpha^{2}(\underline b ,t)|\sqrt{a}\nabla p_t(b,0)|^2 \leq C\frac{\Y_{t}}{(1+t)^{\frac{d}{2}+1}},\label{LG2}
\end{align}
where $\overline{\mathcal{X}}(x):= \mathcal{X}(\tau_x a)$ is the stationary extension of the random variable $\mathcal{X}$ defined in Lemma \ref{MO} and
\begin{align}\label{LG2b}
\Y_{t}:= \X^{\frac 3 2} {(1+t)}^{-\frac{d}{2}}\sum_z \omega_{\frac d 2 +2}^{-2}(z, t)\sqrt{\overline{\X}(z)}.
\end{align}
Moreover, for every $p\in [1, +\infty)$ it holds
\begin{align}\label{LG2c}
\langle |\sup_{t > 0}\Y_t|^p \rangle^{\frac 1 p} \leq C(p,d) < +\infty.
\end{align}
\end{lem}

\medskip

The proof of Lemma~\ref{MO} consists in showing that the environments we consider are ``$w$-moderate'' in the sense defined in \cite{MO}. 

\begin{lem}\label{w-mod}
There exists a family of non-negative random variables $\{ w(e) \}_{e\in \mathbb{B}}$ and a family of nearest-neighbor paths $\{\pi(e)\}_{e \in \mathbb{B}}$ such that the following properties hold:
\begin{itemize}
 \item[(i)] For every $e\in \mathbb{B}^d $ and $q\in [1, +\infty)$
 \begin{align}\label{W2}
  \langle |w(e)|^{-q} \rangle \leq C(d,q) < +\infty;
 \end{align}
 \item[(ii)] Let $\xi: \Omega \rightarrow \R$ be a random variable and $\zeta: \Omega \times \Z^d \rightarrow \R$ a random field; for every $e \in \mathbb{B}^d$, the path $\pi(e)$ connects the two endpoints of $e$, it is such that its length $|\pi(e)|$ satisfies, for every $q\in [1, +\infty)$,
\begin{align}\label{W3}
\langle |\pi(e)|^q \rangle \leq C(d,q) < +\infty,
\end{align}
and it holds
 \begin{align}
 w(e)|D_e\xi(a)|^2 &\le \sum_{b\in\pi(e)} a(b)|D_b \xi(a)|^2 ,\label{W1}\\
 w(e)|\nabla\zeta(a, e)|^2 &\le \sum_{b\in\pi(e)} a(b)|\nabla\zeta(a, b)|^2. \label{W1b}
\end{align}
\item[(iii)] Both $w(\cdot)$ and $\pi(\cdot)$ are stationary;
\end{itemize}
\end{lem}
\begin{proof} 
In this proof the notation $\lesssim$ stands for $\leq C$ with $C=C(d,q)$.
For every edge $e\in \B$, we define
\begin{align}\label{wmod1}
w(e)^{-1}:= \inf\{\sum_{b\in \pi(e)} a^{-1}(b) \ : \ \mbox{$\pi(e)$ connects the two endpoints of $e$} \}.
\end{align}
Since $a^{-1}$ is bounded from below, there exists a path that achieves the infimum above. We choose one according to a fixed, deterministic tie-breaking rule, and denote it by $\pi(e)$. 
With this definition of weights and paths, the point (iii) immediately follows by stationarity of $a$. We also have
\begin{align*}
|D_e\xi|^2 &= |\sum_{b\in \pi(e)}a(b)^{-\frac 1 2}a(b)^{\frac 1 2}D_b\xi|^2\\
& \leq \bigl(\sum_{b\in \pi(e)}a(b)^{-1}\bigr) \bigl(\sum_{b\in \pi(e)}a(b)|D_b \xi|^2 \bigr) \stackrel{(\ref{wmod1})}{=}  w(e)^{-1}\bigl(\sum_{b\in \pi(e)}a(b)|D_b \xi|^2 \bigr),
\end{align*}
i.e. inequality (\ref{W1}). Note that by definition of $\nabla$, an analogous calculation yields \eqref{W1b}. Moreover, since $a^{-1} \ge 1$, we have
\begin{align*}
|\pi(e)| \leq \sum_{b \in \pi(e)}a(b)^{-1},
\end{align*}
and thanks to (\ref{wmod1}), the bound (\ref{W3}) is directly implied by (\ref{W2}). In order to show this last bound, we want to argue that for every $q\in [1, +\infty)$ and $x \gg 1$
\begin{align}\label{wmod3}
\mathbb{P}( w^{-1}(e) > x ) \lesssim x^{-q}.
\end{align}
We proceed in the following way: Thanks to assumption (\ref{a.mom.bounds}) and independence, it holds for $y \in \R$ to be fixed below that
\begin{align*}
\mathbb{P}( (\sup_i a(e_i))^{-1} \geq y)= \prod_{i=1}^d\mathbb{P}( a(e_i)^{-1}\geq y)\lesssim y^{-2qd},
\end{align*}
and therefore there exists a (random) $i=i(y)$ such that
\begin{align}\label{wmod4}
\mathbb{P}( a(e_i)^{-1}\geq y )\lesssim y^{-2q}.
\end{align}
The main idea is to explicitly construct a path $\tilde\pi(e)$, connecting the two endpoints of $e$ for which we have some control on the quantity
$\P( \sum_{b\in \tilde \pi(e)} a(b)^{-1} > x)$:
From that, thanks to definition (\ref{wmod1}), we also obtain the same bound for (\ref{wmod3}).
Without loss of generality, let us assume that $e=(z,z+e_1)$ for some $z\in \Z^d$. Therefore, if $i(y)=1$ we just choose $\tilde \pi(e)= e$ and get by stationarity that for every $x > y$
\begin{align}\label{wmod7}
\P( \sum_{b\in \tilde \pi(e)} a(b)^{-1} > x)  = \mathbb{P}(a(e)^{-1} > x) \stackrel{(\ref{wmod4})}{\lesssim}  y^{-2q},
\end{align}
i.e. the bound (\ref{W2}). If otherwise $i \neq 1$, then by stationarity and our assumption on the random variables $\{ a(b) \}_{b \in \B}$ to be non-degenerate, we may fix a $\varepsilon > 0 $ (independent on $i$ and $x$) and consider
\begin{align*}
K:=\inf\{k \geq 0 \ : \ a((z + k\overline e_i, z+k\overline e_i +\overline e_1) > \varepsilon \},
\end{align*}
which satisfies 
\begin{align}\label{wmod4bis}
\mathbb{P}(K > k) \leq \exp{(-ck)},
\end{align}
for a positive constant $c= c(\varepsilon)$.
Therefore, we estimate for any $x > \frac{2}{\varepsilon}$
\begin{align}\label{wmod5}
\mathbb{P}&(w^{-1}(e) \geq x )=\mathbb{P}(w^{-1}(e) \geq x\ , \ K > k ) + \mathbb{P}(w^{-1}(e) \geq x\ , \ K \leq k ).
\end{align}
We control the first term on the r.h.s of (\ref{wmod5}) by
\begin{align}\label{wmod6}
\mathbb{P}(w^{-1}(e) \geq x\ , &\ K > k )\notag\\
& = \mathbb{P}(w^{-1}(e) \geq x\ \bigl| \ K > k )\ \mathbb{P}(K > k )\stackrel{(\ref{wmod4bis})}{\leq} \exp{(-ck)}.
\end{align}
For the second term the idea is two observe that, if $K \leq k$, then we might consider as path
$$
\tilde\pi(e):= \{\tilde b_1, ..., \tilde b_k , \tilde e , \tilde b_{k+1}, ..., \tilde b_{2k}\}
$$ 
the one starting from z, moving k steps in direction $e_i$, then moving in direction $e_1$ and finally going back with other k steps to $x+e_1$. 
Therefore,
\begin{align*}
\mathbb{P}( w^{-1}(e)& \geq x\ , \ K \leq k )\leq \mathbb{P}(\sum_{b\in \tilde \pi(e)}a^{-1}(b) \geq x\ , \ K \leq k ),
\end{align*}
and since by construction 
$$
|\tilde \pi(e)| = 2k +1 \ \ \text{ and} \ \ \sum_{b\in \tilde \pi(e)} a(b)^{-1} \leq \varepsilon^{-1} \sum_{j=1}^{2k}a(\tilde b_j)^{-1},
$$
we may control
\begin{align*}
\mathbb{P}( w^{-1}(e) \geq x \ &, \ K \leq k )\\
&\leq \mathbb{P}(\mbox{ $\ \exists \ j\in \{1,...,2k\}$ such that $a(\tilde b_j) \geq \frac{x}{4k}$  }, \ K \leq k ).
\end{align*} 
Independence and then stationarity hence yield
\begin{align*}
\mathbb{P}&( w^{-1}(e) \geq x\ , \ K \leq k )\leq 2k \ \mathbb{P}({a(e_i)^{-1} \geq \frac{x}{4k}}).
\end{align*}
Fixing now $k= x^\eta$ with $\eta < < 1$, we get
\begin{align*}
\mathbb{P}&( w^{-1}(e) \geq x\ , \ K \leq x^{\eta} )\leq x^{\eta}\,\mathbb{P}({a(e_i)^{-1} \geq \frac{x^{1-\eta}}{4}}),
\end{align*}
so that if we choose $y= \frac{x^{1-\eta}}{4} < x$ in (\ref{wmod4}), this turns into
\begin{align*}
\mathbb{P}&( w^{-1}(e) \geq x\ , \ K \leq x^{\eta} )\lesssim x^{-q},
\end{align*}
and (\ref{wmod6}) and (\ref{wmod7}) respectively into
\begin{align*}
\mathbb{P}( w^{-1}(e) \geq x\ , \ K > x^{\eta} ) &\leq \exp{(-c x^{\eta} )} \lesssim x^{-q},\\
 \mathbb{P}(a(e)^{-1} > x)  =\mathbb{P}(a(e)^{-1} > x) &\lesssim x^{-2(1-\eta)q} \lesssim x^{-q}
\end{align*}
By wrapping up the previous three inequalities we conclude (\ref{wmod3}) and hence (\ref{W2}).
\end{proof}
\begin{lem}\label{Inv.Pi}
For every $b \in \B$, let
\begin{align}\label{Piinv}
\pi^{-1}(b):= \{ e\in \mathbb{B} \ \ : \ \ b \in \pi(e) \}.
\end{align}
Then, for every $p\in [1, +\infty)$
\begin{equation}\label{Pi1}
\langle |\pi^{-1}(b)|^p \rangle \leq C(d,p) < +\infty .
\end{equation}
\end{lem}
\begin{proof}
For a fixed edge $b\in \mathbb{B}$ and any $p\in [1, +\infty)$, let us consider
\begin{align*}
\mathbb{P}&(|\pi^{-1}(b)| \geq k ) = \mathbb{P}(\exists \ \ e_1, ..., e_{k} \ \ : \ \ \forall i \ \ \ b \in \pi(e_i)\ ).
\end{align*}
We observe that if there are $\sim k$ edges whose optimal path passes through b, then there must be and edge $e$ with $|b-e| \geq k^{\frac {1} {d}}$. Therefore,
\begin{align*}
\mathbb{P}&(|\pi^{-1}(b)| \geq k )\leq \mathbb{P}(\exists \ \ \tilde e \ \text{with}\ \ |\tilde e-b|\geq k^{\frac{1}{d}} \ \ : \ \ \tilde e \in \pi(b)\ ).
\end{align*}
The path $\pi(\tilde e)$ being connected, allows us to estimate
\begin{align*}
\mathbb{P}(|\pi^{-1}(b)|& \geq k )\leq \mathbb{P}(\exists \ \ \tilde e \ \text{with}\ \ |\tilde e-b|\geq k^{\frac{1}{d}} \ \ : \ |\pi(\tilde e)| > |b-\tilde e|\ )\\
&\leq \sum_{e : |b -e|> k^{\frac{1}{d}}} \mathbb{P}(|\pi(e)| > |b-e|) \simeq \sum_{n\sim k^{\frac{1}{d}}}^{+\infty} \sum_{|e-b|=n}\mathbb{P}(|\pi(e)|> n).
\end{align*}
Chebyshev's inequality yields for every $q \in [1; +\infty) $
\begin{align*}
\mathbb{P}&(|\pi^{-1}(b)| \geq k )\leq \sum_{n\sim k^{\frac{1}{d}}}^{+\infty} \sum_{|e-b|=n} n^{-q} \langle |\pi(e)|^q \rangle
\stackrel{(\ref{W3})-(iii)}{\leq} C(d,q) \sum_{n\sim k^{\frac{1}{d}}}^{+\infty} n^{d-1-q}.
\end{align*}
We may now choose q big enough, e.g. $q = 2(p+1) d $, to conclude
\begin{align*}
\mathbb{P}&(|\pi^{-1}(b)| \geq k ) \leq C(d,2(p+1)d) k^{-(2p+1)},
\end{align*}
which implies inequality (\ref{Pi1}) for every $p \in [1 , +\infty)$.
\end{proof}
We now show the following general result on stationary random fields.
\begin{lem}\label{SUM}
Let $\zeta: \Omega \times \mathbb{Z}^d \rightarrow \mathbb{R}$ be a stationary random field.
Then for every $p\in [1,+\infty)$, edges $e_0,b_0 \in \mathbb{B}$ and $\delta >0$ there exists a $C=C(d, p, \delta) < +\infty$ such that
\begin{align}
\langle | \sum_{b\in \pi(e_0)} \zeta(a,\underline b) |^p \rangle^{\frac 1 p} \leq C \langle |\zeta(a,0)|^{p(1+\delta)}\rangle^{\frac{ 1}{p(1+\delta)}},\label{SUM1}\\
\langle | \sum_{e\in \pi^{-1}(b_0)}{\zeta}(a,\underline e) |^p \rangle^{\frac 1 p} \leq C \langle |\zeta(a,0)|^{p(1+\delta)}\rangle^{\frac{ 1}{p(1+\delta)}}\label{SUM2},
\end{align}
whenever the r.h.s.\ of (\ref{SUM1})-(\ref{SUM2}) is finite.
\end{lem}
\begin{proof}
For the sake of simplicity, we skip the argument $a$ in $\zeta$ and write $\lesssim$ instead of $\leq C$, with $C$ depending on $d$, $p$ and $\delta$.
We start with (\ref{SUM2}): Let us fix $p \in [1, +\infty)$ and $\delta >0$. Since by (\ref{Pi1}) we have
$\langle \cdot \rangle$-almost surely that $|\pi^{-1}(z)| < +\infty$, we may write
\begin{align*} 
\langle | \sum_{e\in \pi^{-1}(b_0)} &\zeta(\underline e) |^p \rangle\\
& = \sum_{n=0}^{+\infty} \mathbb{P}(\max_{e\in \pi^{-1}(b_0)}|e-b_0| = n )\, \langle | \sum_{e\in \pi^{-1}(b_0)}{\zeta}(\underline e) |^p \bigr| \max_{e\in \pi^{-1}(b_0)}|e-b_0| = n \rangle \\
&\lesssim \sum_{n=0}^{+\infty} \mathbb{P}(\max_{e\in \pi^{-1}(b_0)}|e-b_0| = n ) \langle |\sum_{e\in B_{n+1}(\underline{b_0})}{\zeta}(\underline e) |^p \bigr| \max_{e\in \pi^{-1}(b_0)}|e-b_0| = n \rangle,
\end{align*}
and by H\"older's inequality in $e$
\begin{align*}
\langle |& \sum_{e\in \pi^{-1}(b_0)}{\zeta}(\underline e) |^p \rangle\\
&\lesssim \sum_{n=0}^{+\infty}(n+1)^{d(p - 1)} \mathbb{P}(\max_{e\in \pi^{-1}(b_0)}|e-b_0| = n )  \sum_{e\in B_{n+1}(\underline{b}_0)} \langle |\zeta(\underline e)|^p \bigr| \max_{e\in \pi^{-1}(b_0)}|e-b_0| = n \rangle.
\end{align*}
We now decompose the second term in the r.h.s. of the previous inequality as
\begin{align*}
\mathbb{P}(\max_{e\in \pi^{-1}(b_0)}|e-b_0| = n )= \mathbb{P}(\max_{e\in \pi^{-1}(b_0)}|e-b_0| = n )^{\frac{\delta}{1+\delta}} \mathbb{P}(\max_{e\in \pi^{-1}(b_0)}|e-b_0| = n )^{\frac{1}{1+\delta}},
\end{align*}
and thus rewrite
\begin{align*}
\langle |& \sum_{e\in \pi^{-1}(b_0)}{\zeta}(\underline e) |^p \rangle\\
&\lesssim \sum_{n=0}^{+\infty}(n+1)^{d(p - 1)}  \mathbb{P}(\max_{e\in \pi^{-1}(b_0)}|e-b_0| = n )^{\frac{\delta}{1+\delta}}\\
&\hspace{0.5cm}\times\mathbb{P}(\max_{e\in \pi^{-1}(b_0)}|e-b_0| = n )^{\frac{1}{1+\delta}}\sum_{e\in B_{n+1}(\underline{b}_0)} \langle |\zeta(\underline e)|^p \bigr| \max_{e\in \pi^{-1}(b_0)}|e-b_0| = n \rangle.
\end{align*}
Therefore, an application of H\"older's inequality with exponents $(1+2\delta, \frac{1+2\delta}{2\delta})$ in $n$ yields
\begin{multline}
\label{SUM3}
\langle | \sum_{e\in \pi^{-1}(b_0)} \overline{\zeta(\underline e)} |^p \rangle 
\lesssim \biggl(\sum_{n=0}^{+\infty}(n+1)^{d(p - 1) \frac{1+2\delta}{2\delta}}  \mathbb{P}(\max_{e\in \pi^{-1}(b_0)}|e-b_0| = n )^{\frac{1+2\delta}{2(1+\delta)}}\biggr)^{\frac{2\delta}{1+2\delta}} \\
\times\biggl( \sum_{n=0}^{+\infty}\mathbb{P}(\max_{e\in \pi^{-1}(b_0)}|e-b_0| = n )^{1+\frac{\delta}{1+\delta}}\bigl(\sum_{e\in B_{n+1}(\underline{b}_0)} \langle |\zeta(\underline e)|^p \bigr| \max_{e\in \pi^{-1}(b_0)}|e-b_0| = n \rangle \bigr)^{1+2\delta} \biggr)^{\frac{1}{1+2\delta}}.
\end{multline}
We now observe that the first term on the r.h.s. of (\ref{SUM3}) may be bounded by 
\begin{align}\label{SUM4}
\biggl( \sum_{n=0}^{+\infty}(n+1)^{d(p - 1) \frac{1+2\delta}{2\delta}}  \mathbb{P}(\max_{e\in \pi^{-1}(b_0)}|e-b_0| = n )^{\frac{1+2\delta}{2(1+\delta)}} \biggr)^{\frac{2\delta}{1+ 2\delta}} \lesssim 1.
\end{align}
This follows after noting that since
\begin{align*}
\mathbb{P}(\max_{e\in \pi^{-1}(b_0)}|e-b_0| = n ) \leq \mathbb{P}(\exists \ \tilde e \ \text{with} \ |\tilde e-b|\geq n \ : \ |\pi(e)|> n ),
\end{align*}
by the same reasoning of Lemma \ref{Inv.Pi} we have for every $M \geq 1$
\begin{align}\label{SUM5}
\mathbb{P}(\max_{e\in \pi^{-1}(b_0)}|e-b_0| = n ) \leq C(d,M) n^{-M},
\end{align}
and thus also (\ref{SUM4}) for $M$ large enough.
We turn to the second term on the r.h.s. of (\ref{SUM3}) and claim that
\begin{align}\label{SUM6}
\sum_{n=0}^{+\infty}\mathbb{P}(\max_{e\in \pi^{-1}(b_0)}|e-b_0| = n )^{1+\frac{\delta}{1+\delta}}\bigl(\sum_{e\in B_{n+1}(\underline{b}_0)} \langle |\zeta(\underline e)|^p \bigr| &\max_{e\in \pi^{-1}(b_0)}|e-b_0| = n \rangle \bigr)^{1+2\delta}\notag\\
&\lesssim \langle |\zeta(0)|^{p(1+2\delta)}\rangle.
\end{align}
Indeed, we write
\begin{align*}
& \sum_{n=0}^{+\infty}\mathbb{P}(\max_{e\in \pi^{-1}(b_0)}|e-b_0| = n )^{1+\frac{\delta}{1+\delta}}\bigl(\sum_{e\in B_{n+1}(\underline{b}_0)} \langle |\zeta(\underline e)|^p \bigr| \max_{e\in \pi^{-1}(b_0)}|e-b_0| = n \rangle \bigr)^{1+2\delta}\\
&= \sum_n \sum_m \1_{n}(m) \mathbb{P}(\max_{e\in \pi^{-1}(b_0)}|e-b_0| = m )^{\frac{\delta}{1+\delta}}\\
&\hspace{1cm}\times\mathbb{P}(\max_{e\in \pi^{-1}(b_0)}|e-b_0| = n)\bigl(\sum_{e\in B_{m+1}(\underline{b}_0)} \langle |\zeta(\underline e)|^p \bigr| \max_{e\in \pi^{-1}(b_0)}|e-b_0| = n \rangle \bigr)^{1 + 2\delta},
\end{align*}
and by H\"older's inequality first in the $e$-variable and then in $\langle \cdot \rangle$
\begin{align*}
&\sum_{n=0}^{+\infty}\mathbb{P}(\max_{e\in \pi^{-1}(b_0)}|e-b_0| = n )^{1+\frac{\delta}{1+\delta}}\bigl(\sum_{e\in B_{n+1}(\underline{b}_0)} \langle |\zeta(\underline e)|^p \bigr| \max_{e\in \pi^{-1}(b_0)}|e-b_0| = n \rangle \bigr)^{1+2\delta}\\
&\leq \sum_n \sum_m \1_{n}(m) m^{d\frac{1+2\delta}{2\delta}}\mathbb{P}(\max_{e\in \pi^{-1}(b_0)}|e-b_0| = m )^{\frac{\delta}{1+\delta}}\\
&\hspace{1cm}\times\mathbb{P}(\max_{e\in \pi^{-1}(b_0)}|e-b_0| = n)\sum_{e\in B_{m+1}(\underline{b}_0)} \langle |\zeta(\underline e)|^{p(1+2\delta)} \bigr| \max_{e\in \pi^{-1}(b_0)}|e-b_0| = n \rangle\\
&\leq \sum_m m^{d\frac{1+2\delta}{2\delta}}\mathbb{P}(\max_{e\in \pi^{-1}(b_0)}|e-b_0| = m )^{\frac{\delta}{1+\delta}}\\
&\hspace{1cm}\times \sum_{e\in B_{m+1}(\underline{b}_0)}\sum_n\mathbb{P}(\max_{e\in \pi^{-1}(b_0)}|e-b_0| = n) \langle |\zeta(\underline e)|^{p(1+2\delta)} \bigr| \max_{e\in \pi^{-1}(b_0)}|e-b_0| = n \rangle\\
&=\sum_m m^{d\frac{1+2\delta}{2\delta}}\mathbb{P}(\max_{e\in \pi^{-1}(b_0)}|e-b_0| = m )^{\frac{\delta}{1+\delta}}\sum_{e\in B_{m+1}(\underline{b}_0)}\langle |\zeta(\underline e)|^{p(1+2\delta)}\rangle.
\end{align*}
Our assumption that $\zeta$ is stationary thus implies
\begin{align*}
&\sum_{n=0}^{+\infty}\mathbb{P}(\max_{e\in \pi^{-1}(b_0)}|e-b_0| = n )^{1+\frac{\delta}{1+\delta}}\bigl(\sum_{e\in B_{n+1}(\underline{b}_0)} \langle |\zeta(\underline e)|^p \bigr| \max_{e\in \pi^{-1}(b_0)}|e-b_0| = n \rangle \bigr)^{1+2\delta}\\
&\simeq\langle |\zeta(0)|^{p(1+2\delta)}\rangle \sum_m m^{d(1+\frac{1+2\delta}{2\delta})}\mathbb{P}(\max_{e\in \pi^{-1}(b_0)}|e-b_0| = m )^{\frac{\delta}{1+\delta}}.
\end{align*}
Reasoning as for (\ref{SUM4}), we conclude inequality (\ref{SUM6}).
Inserting estimates (\ref{SUM4}) and (\ref{SUM6}) in (\ref{SUM3}) yields inequality (\ref{SUM2}), after relabeling $\delta = 2\delta$.

\medskip

We now prove (\ref{SUM1}) in an analogous way: Thanks to assumption (\ref{W3}), it holds the identity
\begin{align*}
\langle | \sum_{e\in \pi^{-1}(b_0)}{\zeta}(\underline e) |^p \rangle= \sum_{n=1}^{+\infty}\mathbb{P}(|\pi(e_0 )|= n)\langle | \sum_{e\in \pi^{-1}(b_0)}{\zeta}(\underline e) |^p \bigl |  |\pi(e_0 )|= n \rangle.
\end{align*}
We now reason exactly as in the argument for (\ref{SUM2}), this time relying directly on (\ref{W3}) and on Chebyshev's inequality to infer the analogous of (\ref{SUM5}), i.e. that for every $M \geq 1$
\begin{align*}
\mathbb{P}(|\pi(e_0)| = n ) \leq C(d,M) n^{-M}.
\end{align*}
\end{proof}

\begin{proof}[Proof of Lemma~\ref{MO}]
Thanks to Lemma \ref{w-mod}, (ii) we may apply Theorem 3.2 of \cite{MO} and obtain that there exist $r >0$ and $q > d$ such that
\begin{align*}
p_t(0,0) \leq t^{-\frac d 2} \biggl(\sup_{r}r^{-d}\sum_{|\overline{e}| \leq r} w^{-q}(e)\biggr)^r ,
\end{align*}
with $w(e)$ defined in Lemma \ref{w-mod}. It thus only remains to prove that 
$$
\X:= \biggl(\sup_{r}r^{-d}\sum_{|\overline{e}| \leq r} w^{-q}(e)\biggr)^r
$$
satisfies \eqref{G0}. This follows from Lemma \ref{w-mod}, (i) and the maximal function estimate (\cite{MO}, Corollary A.2 )
\begin{align*}
\langle |\sup_{r}r^{-d}\sum_{|\overline{e}| \leq r} w^{-q}(e)|^p \rangle^{\frac 1 p} \leq C(p) \langle |w^{-q}|^p \rangle^{\frac 1 p},
\end{align*}
for every $p \in (1, +\infty]$.
Estimate \eqref{G2} follows from \eqref{G1} thanks to the identity
\begin{align*}
p_t(0,0) = \sum_{z} p_{\frac t 2}^2(x,0).
\end{align*}
This, thanks to the symmetry of $p_t$, is in turn a particular case of 
\begin{align}\label{LG10a}
p_t(x,0)= \sum_z p_s (x, z)p_{t-s}(z, 0),
\end{align}
with $x \in \Z^d$ and $s \in (0 , t)$. To show \eqref{LG10a} it suffices to observe that since for every $s>0$ we have
\begin{align*}
\begin{cases}
& \partial_t p_{s+t} + \nabla^* a\nabla p_{s+t} =0 \ \ \mbox{ $t>0$}\\
&p_s(x ,0) =p_{s}(x ,0) ,
\end{cases}
\end{align*}
then the representation formula implies the semigroup property  
\begin{align*}
p_{s+t}(x,0)= \sum_z p_t( x, z)p_s(z, 0),
\end{align*}
which is equivalent to \eqref{LG10a} if we relabel $t=t+s$.
\end{proof}

Before proving Lemma \ref{LG}, we state the following auxiliary result, whose proof we postpone to the appendix.
\begin{lem}\label{MaxF}
Let $\alpha > \frac d 2 + 1 $ and let $\mathcal{Z}=\mathcal{Z}(a)$ be a non-negative random variable such that for every $p \in [1, +\infty)$ 
\begin{align}\label{MFb}
\langle |\mathcal{Z}|^p \rangle \leq C(p) < +\infty.
\end{align}
We then have
\begin{align}\label{MFa}
\langle |\sup_{t>0} \sum_{z \in \mathbb{Z}^d} \omega_\alpha^{-2}(z, t) \mathcal{Z}(\tau_z a)|^p \rangle \leq C(d, p,\alpha) < +\infty,
\end{align}
where the weight $\omega_\alpha$ is defined as in (\ref{LG1}).
\end{lem}
\begin{proof}[Proof of Lemma \ref{LG}]
We prove Lemma \ref{LG} similarly to \cite[Theorem~3]{GNO}. We remark that, in contrast with Theorem 3, we do not need to prove an optimal decay in time for the weighted $\ell^{2p}$-norm in space, and we replace inequality (173) by the
stochastic bound (\ref{G2}) of Lemma~\ref{MO}.\\
We start by observing that by \cite[Proposition~3.3 or 3.4]{MO},  Lemma \ref{MO} implies for $\alpha > \frac{d}{2}+1$
\begin{equation}\label{LG3}
 \sum_{x \in \Z^d} \omega_\alpha^2(x,t) \, p_t^2(x,0)  \lesssim \frac{\X}{t^{ \frac d 2}},
\end{equation}
where here and in the rest of this proof $\lesssim$ stands for $\leq C(d,\alpha)$.

\medskip

We start by upgrading inequality (\ref{LG3}) to the bound (\ref{LG2a}):
Since for every $t > 0$ and $s\in (0, t)$
\begin{align}\label{LG10b}
\omega_\alpha(t,x) \lesssim \omega_{\alpha}(s, x-z ) \omega_\alpha( t-s, z),
\end{align}
we may choose the value $s=\frac t 2$ in (\ref{LG10b}) and in (\ref{LG10a}) of Lemma \ref{MO} and obtain

\begin{align}\label{LG11}
\omega_\alpha(t,x)p_t(x,0) &\lesssim \sum_z \omega_\alpha( \frac t 2, z-x)p_{\frac t 2}( x, z) \omega_\alpha (\frac t 2, z) p_{\frac t 2}(z, 0)\notag\\
&\lesssim \Ll(\sum_z \omega_\alpha^2( \frac t 2, z-x)p_{\frac t 2}^2( x, z) \Rr)^{\frac 1 2}\Ll(\sum_z \omega_\alpha^2 (\frac t 2, z) p_{\frac t 2}^2( z, 0) \Rr)^{\frac 1 2}.
\end{align}
By symmetry  and stationarity of $p_t$, it holds 
$$
p_t(a,x,z)= p_t(a,z,x)=p_t(\tau_xa, z-x,0),
$$
so that inequality (\ref{LG11}) turns into
\begin{align*}
\omega_\alpha(x,t)p_t(x,0) &\lesssim \Ll(\sum_z \omega_\alpha^2( \frac t 2, z-x)p_{\frac t 2}^2(\tau_x a, z-x, 0) \Rr)^{\frac 1 2}\Ll(\sum_z \omega_\alpha^2 (\frac t 2, z) p_{\frac t 2}^2(a, z, 0) \Rr)^{\frac 1 2}\\
&\stackrel{(\ref{G2})}{\lesssim} \frac{\sqrt{\mathcal{X}(a)\mathcal{X}(\tau_x a)}}{t^{\frac d 2}}.
\end{align*}
Recalling our definition of stationary extension of a random variable, the previous inequality yields (\ref{LG2a}).

\medskip

In order to get also (\ref{LG2}), we first claim that for every $T >0$
\begin{equation}\label{LG7}
\fint_T^{2T}\sum_b \omega_\alpha^2(\underline b,t) |\sqrt{a(b)}\nabla p_t(b,0)|^2  \lesssim \frac{\X}{ T^{\frac d 2 +1}}.
\end{equation}
The identity
\begin{align*}
\frac{d}{dt}\biggl( \sum_x \frac{1}{2}p_t^2(x,0) \biggr) \stackrel{(\ref{E2})}{=} -\sum_b \nabla p_t(b,0)\cdot a(b) \nabla p_t(b,0)
\end{align*}
implies by integrating
\begin{align*}
\fint_T^{2T}\sum_b \nabla p_t\cdot a\nabla p_t(b,0) \lesssim T^{-1}\sum_x p_T^2(x,0)\stackrel{(\ref{G2})}{\lesssim} \frac{\mathcal{X}}{T^{\frac d 2 +1}}.
\end{align*}
Therefore, as 
\begin{align*}
{\omega}_{\alpha}^{2}(\underline b, t)\nabla p_t\cdot a \nabla p_t(b,0) \lesssim t^{-\alpha}|\underline b|^{2\alpha}\nabla p_t\cdot a \nabla p_t(b,0)+ \nabla p_t\cdot a \nabla p_t(b,0),
\end{align*}
in order to show (\ref{LG7}) it suffices to prove that
\begin{align}\label{L1.7}
\fint_T^{2T}\sum_b |\underline b|^2 \nabla p_t(b,0)\cdot a(b) \nabla p_t(b,0) \lesssim  \frac{\X}{T^{-\alpha+\frac d 2 +1}}.
\end{align}
We write
\begin{align*}
\frac{d}{dt}&\Ll(\sum_{x} |x|^{2\alpha}p_t^2(x, 0)\Rr)\stackrel{(\ref{E2})}{=} - \sum_b \nabla\Ll( |\cdot|^{2\alpha} p_t(\cdot,0)\Rr)(b) \cdot a(b) \nabla p_t(b,0)\\
\lesssim & - \sum_b |\underline b|^{2\alpha} \nabla p_t\cdot a\nabla p_t(b,0) + \sum_b  p_t(\overline b,0)\bigl(|\overline b|^{2\alpha}-|\underline b|^{2\alpha} \bigr) |a(b) \nabla p_t(b,0)|\\
 & \ls  - \sum_b |\underline b|^{2\alpha} \nabla p_t\cdot a\nabla p_t(b,0) + \sum_b  p_t(\overline b,0)|\overline b|^{\alpha-1}|\underline b|^{\alpha} |
a(b) \nabla p_t(b,0)|,
\end{align*}
and thus by H\"older's inequality and Young's inequality we get
\begin{align*}
\frac{d}{dt}\Ll(\sum_{x} |x|^{2\alpha}p_t^2(x, 0)\Rr)
&{\lesssim} - \sum_b |\underline b|^{2\alpha} \nabla p_t\cdot a \nabla p_t(b,0) + \sum_b  |\overline{b}|^{2(\alpha-1)} p^2_t(\overline{b},0)\\
&\simeq - \sum_b |\underline b|^{2\alpha} \nabla p_t\cdot a \nabla p_t(b,0) + \sum_x  |x|^{2(\alpha-1)} p^2_t(x,0) .
\end{align*}
Integrating this inequality in $t\in (T, 2T)$ we obtain
\begin{align*}
\fint_T^{2T}\sum_b |\underline b|^{2\alpha}& \nabla p_t \cdot a \nabla p_t(b,0)\\
 &\lesssim T^{-1}\biggl(\sum_{x} |x|^{2\alpha}p_T^2(x, 0)\biggr) + \fint_{T}^{2T}\sum_x  |x|^{2(\alpha-1)}p_t^2(x,0)
\end{align*}
and therefore 
\begin{align*}
\fint_T^{2T}\sum_b |\underline b|^{2\alpha}& \nabla p_t \cdot a \nabla p_t(b,0)\\
 &\lesssim T^{-1+ \alpha}\sum_{x} \omega_\alpha^{2}(T, x)p_T^2(x, 0) + \fint_{T}^{2T} t^{\alpha - 1}\sum_x  \omega_\alpha^{2}(t,x)p_t^2(x,0).
\end{align*}
Using (\ref{LG3}) we conclude (\ref{L1.7}).

\medskip 

\noindent We finally prove (\ref{LG2}): by (\ref{LG10a}) it holds for every $b\in \mathbb{B}$
\begin{align*}
\sqrt{a}\nabla p_T(b, 0)= \fint_{\frac T 3}^{\frac 2 3 T} \sum_z \sqrt{a}\nabla p_t( b, z) p_{T-t}( z, 0) \, dt,
\end{align*}
so that 
\begin{align}\label{LG10}
\nabla p_T \cdot a \nabla p_T(b,0)&=|\sqrt{a}\nabla p_T(b, 0)|^2 \notag \\
&= |\fint_{\frac T 3}^{\frac 2 3 T} \sum_z \sqrt{a}\nabla p_t(b, z) p_{T-t}( z, 0) |^2  \, dt \notag\\
& \le \fint_{\frac T 3}^{\frac 2 3 T} \sum_z p_{T-t}( z, 0) |\sqrt{a}\nabla p_t(b, z)|^2  \, dt,
\end{align}
where for the last line we appeal to Jensen's inequality in t and, thanks to $\sum_z p_{s}(z,0)=1$ for every $s > 0$, also in $z$ . By (\ref{LG10b}) we write

\begin{align*}
& \sum_b \omega_\alpha^2(\underline b,T)\nabla p_T\cdot a \nabla p_T( b,0)\\\
& \ \stackrel{(\ref{LG10b})}{\lesssim} \fint_{\frac T 3}^{\frac 2 3 T} \sum_z \omega_\alpha^2( T-t, z) p_{T-t}( z, 0) \sum_b \omega_{\alpha}^2(t, \underline b-z )|\sqrt{a}\nabla p_t(b, z)|^2  \, dt \\
& \stackrel{(\ref{LG2a}), \ \beta > d}{\lesssim} \fint_{\frac T 3}^{\frac 2 3 T} \sum_z \sqrt{\overline{\X}(0)\overline{\X}(z)}(T-t)^{-\frac d 2}\omega_\beta^{-2}( T-t, z) \sum_b \omega_{\alpha}^2(t,\underline b-z)|\sqrt{a}\nabla p_t(b, z)|^2 \, dt\\
&\quad \lesssim  T^{-\frac d 2} \sum_z \sqrt{\overline{\X}(0)\overline{\X}(z)}\omega_\beta^{-2}( T, z) \fint_{\frac T 3}^{\frac 2 3 T}\sum_b \omega_{\alpha}^2( t, \underline b-z )|\sqrt{a}\nabla p_t(b, z)|^2 \, dt\\
&\ \stackrel{(\ref{LG7})}{\lesssim}\frac{\X}{T^{d+1}}\sum_z \sqrt{\overline{\X}(0)\overline{\X}(z)}\omega_\beta^{-2}( T, z),
\end{align*}
which implies (\ref{LG2}) if we choose $\beta = \frac d 2 + 2$.

\medskip

To show inequality (\ref{LG2c}) we observe that thanks to (\ref{G0}) it is enough to prove that for every $p \in [1, +\infty)$
\begin{align}\label{MF}
\langle |\sup_{t> 0}{(1+ t)}^{-\frac d 2} \sum_{z} \omega_{\frac d 2 + 2}^{-2}(t, z) \sqrt{\overline{\X}(z)}|^p \rangle \leq C(d, p) < +\infty.
\end{align}
This immediately follows from Lemma \ref{MaxF} since $ \sqrt{\overline{\X}(z)}$ stands for $\sqrt{\overline{\X}(a,z)}= \sqrt\X(\tau_z a)$ and we have chosen $\alpha= \frac d 2 + 2 > \frac d 2 +1$.
\end{proof}
\section{Proof of Theorem \ref{t.main}}
\label{s.main}
Before giving the argument for Theorem \ref{t.main} we introduce two technical results.
The first is a generalization of Lemma 15 of \cite{GNO}.
\begin{lem}\label{ODE}
Assume that for $C_0 > 0$,
\begin{align}
&0 \le a(t)  \le C_0 \biggl( (1+t)^{{ -\gamma_0}} + \int_0^t (1+t-s)^{-\gamma}\, b^\beta(s) \, d s \biggr),\label{ODE1}\\
&0\le b^p(t)  \le -\frac{d}{dt} [a^p(t)] \label{ODE2},
\end{align}
with $p \in [2, +\infty)$, { $\gamma \in [1 , +\infty)$, $\gamma_0 \in (0, \gamma]$} and $\beta \in (\frac{\gamma}{\gamma+\frac 1 p} , 1)$. Then there exists a constant $C=C(\lambda, \beta, p, C_0) < +\infty$ such that
\begin{align}\label{ODE7}
 a(t) \leq C (1+t)^{{ - \gamma_0}}.
\end{align}
\end{lem}
\begin{proof} In this proof we use the notation $\lesssim$ for $\leq C(\lambda,\beta, p, C_0)$.
We define
\begin{align}\label{ODE4}
\Lambda(t) := \sup_{s \le t} (1+s)^{\gamma} a(s), \ \ \ { \Lambda_0(t) := \sup_{s \le t} (1+s)^{\gamma_0} a(s)}
\end{align}
By H\"older's inequality, for any $0 \le t_1 \le t_2$ it follows
\begin{align}\label{ODE3}
\int_{t_1}^{t_2} b^\beta(s) \, ds & \le \left( \int_{t_1}^{t_2} b^p(s) \, ds \right)^{\frac \beta p} (t_2 - t_1)^{1-\frac \beta p} &\stackrel{(\ref{ODE2})}{\le} \left( a^p(t_1) - a^p(t_2)\right)^{\frac \beta p} (t_2 - t_1)^{1-\frac \beta p}\notag\\
&\lesssim a^{\beta}(t_1) \ (t_2 - t_1)^{1-\frac \beta p},
\end{align}
where in the last inequality we use the fact that by (\ref{ODE2}), the function $a^p(t)$ is monotone non-increasing.
Moreover, letting $N\in \mathbb{N}$ be such that $2^{N-1}t_1\leq t_2 \le 2^{N}t_1$, we get
\begin{align*}
\int_{t_1}^{t_2} b^\beta(s) \, ds \le \sum_{n=1}^N \int_{2^{n-1}t_1}^{2^{n}t_1} b^\beta(s) \, ds \stackrel{(\ref{ODE3})}{\lesssim} 
\sum_{n=1}^N t_1^{1-\frac \beta p}2^{(n-1)(1-\frac \beta p)}\  a^{\beta}(2^{n-1}t_1),
\end{align*}
which by the definition of $\Lambda$ in (\ref{ODE4}) implies
\begin{align}\label{ODE5}
\int_{t_1}^{t_2} b^\beta(s) \, ds& \lesssim   t_1^{1-\frac \beta p} \Lambda^{\beta}(t_2)\sum_{n=1}^N 2^{(n-1)(1-\frac \beta p)}(2^{n-1}t_1)^{-\gamma\beta}\notag\\
 & \lesssim  t_1^{1-\beta (\frac 1 p+\gamma)} \Lambda^{\beta}(t_2)\sum_{n=1}^N 2^{(n-1)(1-\beta(\frac 1  p + \gamma))}\lesssim \Lambda^{\beta}(t_2)t_1^{1-\beta (\frac 1 p+\gamma)},
\end{align}
since we assume $\beta (\frac 1 p + \gamma) > \gamma \ge 1$.
Moreover, since by \eqref{ODE2} the function $a$ is non-increasing, we have
\begin{align*}
a(t) &\  \le \frac{2}{t}\int_{\frac{t}{2}}^t a(r) \, d r \\
&\stackrel{(\ref{ODE1})}{\lesssim} \frac{2}{t}\int_{\frac{t}{2}}^t(1+r)^{{ -\gamma_0}} \, d r+ \frac{2}{t}\int_{\frac{t}{2}}^t d r \int_0^r (1+r-s)^{-\gamma}\, b^\beta(s) \, d s\\
&\  \lesssim  (1+t)^{{ -\gamma_0}} + \frac{2}{t}\int_{\frac{t}{2}}^t d r \int_0^r (1+r-s)^{-\gamma}\, b^\beta(s) \, d s.
\end{align*}
We let $\tau \in [0,\frac t 4]$ to be chosen later, and write
\begin{align*}
a(t)\lesssim  (1+t)^{{ -\gamma_0}} &+ \frac{2}{t}\int_{\frac{t}{2}}^t d r \int_0^\tau (1+r-s)^{-\gamma}\, b^\beta(s) \, d s\\
&+\frac{2}{t}\int_{\frac{t}{2}}^t d r \int_\tau^{\frac r 2}(1+r-s)^{-\gamma}\, b^\beta(s) \, d s\\
&+\frac{2}{t}\int_{\frac{t}{2}}^t d r \int_{\frac r 2}^r (1+r-s)^{-\gamma}\, b^\beta(s) \, d s.
\end{align*}
We estimate each of the three last terms in turn:
\begin{align*}
\frac{2}{t}\int_{\frac{t}{2}}^t&  d r \int_0^\tau (1+r-s)^{-\gamma}\, b^\beta(s) \, d s \lesssim (1+t)^{-\gamma}\int_0^\tau b^\beta(s) \, d s\\
& \stackrel{(\ref{ODE3})}{\lesssim} (1+t)^{-\gamma} \tau^{1-\frac \beta p} \  a^{\beta}(0) \lesssim (1+t)^{-\gamma} \tau^{1-\frac \beta p};
\end{align*}
\begin{align*}
\frac{2}{t}\int_{\frac{t}{2}}^t&  d r \int_\tau^\frac{r}{2} (1+r-s)^{-\gamma}\, b^\beta(s) \, d s\\
&\lesssim (1+t)^{-\gamma} \int_\tau^\frac{t}{2} b^\beta(s) \, d s \stackrel{(\ref{ODE5})}{\lesssim} (1+t)^{-\gamma} \frac{\Lambda^\beta(t)}{\tau^{\beta\left( \frac 1 p + \gamma \right) - 1}};
\end{align*}
\begin{align*}
\frac{2}{t}\int_{\frac{t}{2}}^t& d r \int_\frac{r}{2}^r (1+r-s)^{-\gamma}\, b^\beta(s) \, d s\\
& \lesssim \frac{2}{t}  \int_{\frac{t}{4}}^t b^\beta(s) \, d s\int_s^t (1+r-s)^{-\gamma} \, d r\\
&\stackrel{\gamma \geq 1}{\lesssim} \frac{2}{t}\int_{\frac{t}{4}}^t b^\beta(s) \, d s \int_0^{t-s} (1+r)^{-1}\, d r \,\\
&\lesssim \biggl(\int_{\frac{t}{4}}^t b^\beta(s) \, d s \biggr) \biggl( \frac{2}{t}\int_0^t (1+r)^{-1}\, d r \biggr)\\
&\lesssim  \frac{\log(t+1)}{t} \int_{\frac{t}{4}}^t \, b^\beta(s) \, d s \stackrel{(\ref{ODE5})}{\lesssim} \frac {\log(t+1)} {t} \frac{\Lambda^\beta(t)}{(1+t)^{\beta\left( \frac 1 p + \gamma \right) -1}}\\
&\lesssim \frac{\log(1+ t)}{(t+1)^{\beta(\frac 1 p + \gamma) - \gamma} }\frac{\Lambda^\beta(t)}{(1+t)^{\gamma}}.
\end{align*}
Everything together implies
\begin{align*}
(1+t)^{\gamma}a(t) \lesssim (1 + t)^{{  \gamma_0 - \gamma}}+\tau^{1-\frac \beta p} + \Ll(\frac{1}{\tau^{\beta\left( \frac 1 p + \gamma \right) - 1}} + \frac{\log(t+1)}{(t+1)^{\beta(\frac 1 p + \gamma) - \gamma}}\Rr)\Lambda^\beta(t),
\end{align*}
where the implicit constant does not depend on our choice of $\tau \in [0,\frac t 4]$.
It follows from our assumption on $\beta$ and $\gamma$ that we can find $\tau_0 = \tau_0(\beta, \gamma, p) < \infty$ and $C(\beta,\gamma,p) < \infty$ such that for $\frac t 4 \geq \tau= \tau_0$,{ 
\begin{align*}
(1+t)^{\gamma_0}a(t) &\le C \Ll(1 + \tau_0^{1-\frac \beta p}(1+ t)^{\gamma_0-\gamma}\Rr) + (1+ t)^{\gamma_0-\gamma}\frac{1}{2}\Lambda^\beta(t).
\end{align*}
Moreover, by $\gamma_0 \leq \gamma$ and $\beta < 1$ it holds
\begin{align}\label{ODE6} 
(1+t)^{\gamma_0}a(t) &\le C \Ll(1 + \tau_0^{1-\frac \beta p}\Rr) + \frac{1}{2}\Lambda_0(t).
\end{align}
}
Hence, for such $t$, we have
\begin{align*}
&\Lambda_0(t) \le \sup_{0\leq s \le \tau_0}(1+s)^{\gamma_0}a(s) + \sup_{\tau_0 \le s \le t}(1+s)^{\gamma_0}a(s)\\
&\stackrel{(\ref{ODE2}), (\ref{ODE6}), (\ref{ODE1})}{\le}  (1+\tau_0)^{\gamma_0} + C \Ll(1 + \tau_0^{1-\frac \beta p}\Rr) + \frac{1}{2}\Lambda_0(t),
\end{align*}
and this completes the proof.
\end{proof}
\begin{lem}\label{Mono}
Let $u_t$ solve (\ref{E1}). Then, for every integer $p \geq 1$, we have
\begin{align}\label{M1}
 \langle |Du_t \cdot a Du_t|^p \rangle \leq C \biggl( -\frac{d}{dt} \langle u^{2p}(t) \rangle \biggr),
\end{align}
where $C= C(d, p) < +\infty$.
\end{lem}
\begin{proof} 
We estimate
\begin{align*}
\frac{d}{dt} \langle u^{2p}(t) \rangle = - 2p\langle Du^{2p-1}\cdot a Du \rangle
\end{align*}
and we are done once that we prove
\begin{align}\label{M2}
\langle Du^{2p-1}\cdot a Du \rangle \gtrsim \langle |Du\cdot a Du|^p \rangle,
\end{align}
where here and in the rest of this proof we write $\lesssim$ and $\gtrsim$ respectively for $\leq C$ and $\geq C$ with $C$ depending on $d$ and $p$.
We note that the previous inequality can be rewritten as
\begin{align*}
 \sum_{i=1}^d \langle D_iu^{2p-1}\cdot a(e_i) D_iu \rangle \gtrsim \langle |\sum_{i=1}^d D_iu\cdot a(e_i) D_iu|^p \rangle 
\end{align*}
and, since the argument $D_i u \cdot a(e_i) D_i u \geq 0$, it is implied by
\begin{align*}
 \sum_{i=1}^d \langle D_iu^{2p-1}\cdot a(e_i) D_iu \rangle \gtrsim \sum_{i=1}^d \langle | D_iu\cdot a(e_i) D_iu|^p \rangle
\end{align*}
We thus reduce ourselves to prove for every fixed $i= 1, ... , d$
\begin{align*}
 D_iu^{2p-1}\cdot a(e_i) D_iu \gtrsim  | D_iu \cdot a(e_i) D_iu|^p.
\end{align*}
By our assumptions on the coefficients $0 \leq a(e_i) \leq 1$, we observe that if $a(e_i)=0$, then the previous inequality is trivial. If otherwise $a(e_i)\neq 0$, for every $a\in \R$ we have that
\begin{align*}
 (a^{2p-1}-1)a(e_i)(a-1) \stackrel{a(e_i) \leq 1}{\gtrsim} (a^{2p-1}-1)a(e_i)^{2p}(a-1),
\end{align*}
and thus (\ref{M2}) is implied if we show that
\begin{align}\label{M3}
(a^{2p-1}-1)(a-1) \gtrsim (a-1)^{2p},
\end{align}
for every $a\in \mathbb{R}$. We may now argue analogously to \cite{GNO}, Lemma 14, inequality (94) and show (\ref{M3}). 

\end{proof}

\begin{proof}[Proof of Theorem \ref{t.main}]
Throughout this proof, the notation $\lesssim$ stands for $\leq C$ with the constant depending on $d$ and $p$. 
We give ourselves an independent copy $(\td a(e))_{e \in \B}$ of the environment $(a(e))_{e \in \B}$, with the same law and defined on the same probability space. For each given $e \in \B$, we define the environment $a^e$ by
\begin{equation}
\label{e.def.ae}
a^e(b) := 
\begin{cases}
a(b) & \text{if } b \neq e, \\
\td a(e) & \text{if } b = e.
\end{cases}
\end{equation}
In other words, the environment $a^e$ is obtained from the environment $a$ by resampling the conductance at the edge $e$. 

By the independence assumption on the conductances~$(a(e))$, every random variable $f\in L^2(\Omega)$ satisfies the Spectral Gap (or Efron-Stein) inequality
\begin{align}\label{SG}
\Ll\langle \Ll(f - \langle f \rangle \Rr)^2 \Rr\rangle \le \frac 1 2 \sum_{e \in \B} \Ll\langle  \Ll(\partial_e f \Rr)^2 \Rr\rangle,
\end{align}
with respect to the Glauber derivative
$$
\partial_ef(a):= f(a^e)-f(a).
$$
We start by observing that, analogously to Lemma 11 of \cite{GNO}, we can upgrade the spectral gap inequality (\ref{SG}) to
\begin{equation}\label{pSG}
\Ll\langle \Ll|f- \Ll\langle f \Rr\rangle \Rr|^{2p}\Rr\rangle^{\frac 1 p} \lesssim \left\langle \biggl(\sum_{e} \bigl( {\partial_{e}}f\bigr)^{2}\biggr)^{p}\right\rangle^{\frac{1}{p}},\ \ \ \mbox{for integer $p \in [1, +\infty)$.}
\end{equation}

\medskip

Whenever no ambiguity occurs, we write $u_t(x):=\bar u_t(a, x)$ and/or skip the argument $a$ in $u$ and all the random variables involved. We now argue that, appealing to (\ref{pSG}), for $u_t$ solution of (\ref{E1}) it holds for every integer $p \in [1, +\infty)$ that
\begin{align}\label{RF}
 \langle |u_t|^{2p}\rangle^{{ \frac{1}{2p}}} \lesssim \langle \biggl(\sum_{e} \bigl( \sum_{z} p_t(z,0)&\partial_{e}\overline{g}(z) \bigr)^{2}\biggr)^{p}\rangle^{{ \frac{1}{2p}}}\notag\\
 + &\int_0^t \langle\biggl(\sum_b |\nabla p_{t-s}(b,0)|^2| \nabla \bar u_s(a^{e_i},b)|^2 \biggr)^{p}\rangle^{{ \frac{1}{2p}}},
\end{align}
with $\bar u_t$, and $\bar g$ the stationary extensions solving (\ref{E2}).\\
To show the previous bound we take the Glauber derivative $\partial_{e}$ in (\ref{E2}) and, thanks to the relation
$$
[\partial_e,\nabla] := \partial_e\nabla - \nabla\partial_e  = 0,
$$
we obtain the parabolic boundary value problem
\begin{equation}\label{T1}
\begin{cases}
 & \partial_t\partial_{e}\bar u_t+ \nabla^*a \nabla\partial_{e}{\bar u}_t=-\nabla^*h_t \\
 & \partial_{e} u_0 = \partial_{e}\bar g ,
\end{cases}
\end{equation}
where $h_t= h_t(a, b,e)$ is defined as
\begin{align}\label{T2}
h_t(a, b,e)&:=\partial_{e}a(b) \nabla u_t( a^{e},b) =  \1_{b=e} \partial_e a(e) \nabla \bar u_t(a^{e},e).
\end{align}
Using Duhamel's formula, equation (\ref{T1}) yields
\begin{align*}
\partial_{e}\overline{u}_t(x)=&\sum_{z} p_t(z,x)\partial_{e}\overline{g}(z)+\int_0^t \biggl( \sum_b \nabla p_{t-s}(b,x) h_s(b,e)\biggr)ds,
\end{align*}
which can be rewritten thanks to (\ref{T2}) as
\begin{align}\label{T4}
\partial_{e}\overline{u}_t(x)=&\sum_{z} p_t(z,x)\partial_{e}\overline{g}(z)+\int_0^t \biggl( \nabla p_{t-s}(e,x) \partial_e a(e) \nabla \bar u_s(a^{e},e)\biggr)ds.
\end{align}
Furthermore, we note that for a general random variable $\zeta=\zeta(a)$ it holds
\begin{align}\label{T9}
\partial_{e}\bar{\zeta}(a,x)&= \zeta( (\tau_x a)^{e})-\zeta(\tau_x a)\notag\\
&=\zeta(\tau_x(a^{e-x})) - \zeta(\tau_x a)= \overline{\partial_{e-x}\zeta}(a,x) ,
\end{align}
so that it follows for $u_t$ that 
\begin{align*}
\partial_{e}\overline{u}_t(a,0) = \overline{\partial_{e}u}_t(a,0).
\end{align*}
Taking in the previous identity the $2p$-th power and the average, stationarity implies that
\begin{align}\label{T5}
\langle |\partial_{e}\overline{u}_t(0)|^{2p} \rangle = \langle |\partial_{e}u_t|^{2p} \rangle.
\end{align}
Since by (\ref{E1}) for both assumptions (a) and (b) on the initial data we have $\langle u_t \rangle =0$, we may plug the identities (\ref{T5}) and (\ref{T4}) into the p-Spectral Gap (\ref{pSG}) and get
\begin{align}\label{T6}
 \langle |u_t|^{2p}\rangle^{{ \frac{1}{2p}}}&\lesssim \langle \biggl(\sum_{e} \bigl( {\partial_{e}}\overline{u}_t(0)\bigr)^{2}\biggr)^{p}\rangle^{{ \frac{1}{2p}}}\notag \\
 &\lesssim \langle \biggl(\sum_{e} \bigl( \sum_{z} p_t(z,0)\partial_{e}\overline{g}(z) \bigr)^{2}\biggr)^{p}\rangle^{{ \frac{1}{2p}}}\notag\\
& \hspace{3cm}+ \langle \biggl(\sum_{e} \bigl( \int_0^t  \nabla p_{t-s}(e,0) \partial_e a(e) \nabla \bar u_t(a^{e},e)ds \bigr)^{2}\biggr)^{p}\rangle^{{ \frac{1}{2p}}}.
\end{align}
We now apply the triangle inequality together with (\ref{T2}) on the second term on the r.h.s and estimate
\begin{align*}
\langle \biggl(\sum_{e} \bigl( \int_0^t  \nabla p_{t-s}(e,0) \partial_e a(e) &\nabla \bar u_t(a^{e},e)ds \bigr)^{2}\biggr)^{p}\rangle^{{ \frac{1}{2p}}}\\
&\lesssim \int_0^t \langle \biggl(\sum_e |\nabla p_{t-s}(e,0)|^2| \nabla \bar u_s(a^{e},e)|^2 \biggr)^{p}\rangle^{{ \frac{1}{2p}}}\, ds,
\end{align*}
i.e. inequality (\ref{RF}).

\medskip 

\noindent We are now ready to prove part (a): We start focussing on the first term on the r.h.s. of (\ref{RF}) and claim that
\begin{align}\label{T16}
\langle \biggl(\sum_{e} \bigl( \sum_{z} p_t(z,0)\partial_{e}\overline{g}(z) \bigr)^{2}\biggr)^{p}\rangle^{\frac{1}{p}}\lesssim t^{-\frac d 2} N^{  2} ||g||^2_{L^{\infty}(\Omega)}.
\end{align}
We rewrite the l.h.s of (\ref{T16}) as
\begin{align*}
\langle \biggl(\sum_{y\in \Z^d \atop i=1,... , d} \bigl( \sum_{z} p_t(z,0)&\partial_{\{y, y+e_i \}}\overline{g}(z) \bigr)^{2}\biggr)^{p}\rangle^{\frac{1}{p}}\\
\lesssim  \sum_{i=1,... , d}\langle &\biggl(\sum_{y} \bigl( \sum_{z} p_t(z,0)\partial_{\{y, y+e_i \}}\overline{g}(z) \bigr)^{2}\biggr)^{p}\rangle^{\frac{1}{p}}.
\end{align*}
For every fixed $i$, we estimate by the triangle inequality
\begin{align}\label{T15}
\langle \biggl(\sum_{y} \bigl( \sum_{z} p_t(z,0)&\partial_{\{y, y+e_i\}}\overline{g}(z) \bigr)^{2}\biggr)^{p}\rangle^{{ \frac{1}{2p}}}\le \sum_{z} \langle \biggl( \sum_{y} p_t^2(z,0)|\partial_{\{y, y+e_i\}}\overline{g}(z)|^2 \biggr)^{p}\rangle^{{ \frac{1}{2p}}} \notag\\
&\stackrel{(\ref{T9})}{=}\sum_{z} \langle \biggl(\sum_{y}   p_t^2(z,0)|\overline{\partial_{\{y-z, y-z+e_i\}}g}(z)|^2 \biggr)^{p}\rangle^{{ \frac{1}{2p}}}\notag\\
&\stackrel{z=y-z}{=}\sum_{z} \langle \biggl(\sum_{y} p_t^2(y-z,0)|\overline{\partial_{\{z,z+e_i\}}g}(y-z)|^2 \biggr)^{p}\rangle^{{ \frac{1}{2p}}}\notag\\
&\stackrel{y=y-z}{=}\sum_{z} \langle \biggl(\sum_{y} p_t^2(y,0)|\overline{\partial_{\{z,z+e_i\}}g}(y)|^2 \biggr)^{p}\rangle^{{ \frac{1}{2p}}}\notag\\
&\stackrel{(\ref{LG2a})}{\lesssim}\sum_{z} \langle \biggl(\sum_{y} {\overline{\X}(0)\overline{\X}(y)}t^{-d} \omega_{\alpha}^{-2}(t,y) |\overline{\partial_{\{z,z+e_i\}}g}(y)|^2 \biggr)^{p}\rangle^{{ \frac{1}{2p}}}.
\end{align}
We now use in the last term the triangle inequality in the inner sum and $\langle \cdot \rangle$ and infer that
\begin{align*}
\langle \biggl(\sum_{y} \bigl( \sum_{z} p_t(z,0)\partial_{\{y, y+e_i\}}&\overline{g}(z) \bigr)^{2}\biggr)^{p}\rangle^{{ \frac{1}{2p}}}\\
&\lesssim t^{-\frac{d}{ 2}} \sum_{z} \biggl(\sum_{y}\omega_{\alpha}^{-2}(t,y) \langle \bigl(\overline{\X}(0)\overline{\X}(y) \bigr)^{p}|\overline{\partial_{\{z,z+e_i\}}g}(y)|^{2p} \rangle^{\frac{1}{p}}\biggr)^{{ \frac 1 2}}.
\end{align*}
After a repeated application of H\"older's inequality in $\langle \cdot \rangle$, stationarity and (\ref{G0}) yield
\begin{align*}
\langle \biggl(\sum_{y} \bigl( \sum_{z} p_t(z,0)\partial_{\{y, y+e_i\}}\overline{g}(z)& \bigr)^{2}\biggr)^{p}\rangle^{{ \frac{1}{2p}}}\\
&\lesssim t^{-\frac{d}{ 2}} \sum_{z}|| \partial_{\{z,z+e_i\}}g||_{L^{\infty}(\Omega)} \biggl(\sum_{y} \omega_{\alpha}^{-2}(t,y)\biggr)^{ \frac 1 2}\\
&\stackrel{\alpha > \frac d 2}{\lesssim} t^{-\frac {d}{  4}}  \sum_{z}|| \partial_{\{z,z+e_i\}}g||_{L^{\infty}(\Omega)}.
\end{align*}
Appealing to our assumption on $g$ to be in $L^\infty(\Omega)$ and depending on $N$ edges, we get that
\begin{align*}
\langle \biggl(\sum_{y} \bigl( \sum_{z} p_t(z,0)\partial_{\{y, y+e_i\}}\overline{g}(z)& \bigr)^{2}\biggr)^{p}\rangle^{\frac{1}{p}} \lesssim N^2 ||g||^2_{L^{\infty}(\Omega)} t^{-\frac d 2} .
\end{align*}
Summing over $i=1, ..., d$ yields (\ref{T16}). 

\medskip

We now turn to the second term of the r.h.s. of (\ref{RF}) to argue that for every $\delta> 0$ it holds
\begin{align}\label{T17}
\int_0^t \langle\biggl(\sum_e |\nabla p_{t-s}(e,0)|^2|& \nabla \bar u_s(a^{e},e)|^2 \biggr)^{p}\rangle^{{ \frac{1}{2p}}}\, ds\notag\\
\lesssim &C(\delta) \int_0^t (t-s)^{-(\frac d { 4} +{ \frac 1 2})(1 -\frac 1 p)+ \frac{ d}{{ 2} p}} \langle |D u_t \cdot a D u_t|^{p(1+\delta)} \rangle^{\frac{1}{{ 2}p(1+\delta)}} \, ds .
\end{align}
By \eqref{W1b} of Lemma \ref{w-mod}, we write
\begin{align*}
&\int_0^t \langle\biggl(\sum_e |\nabla p_{t-s}(e,0)|^2| \nabla \bar u_s(a^{e},e)|^2 \biggr)^{p}\rangle^{{ \frac{1}{2p}}}\, ds \\
&\lesssim \int_0^t \langle\biggl(\sum_e w(e)^{-1}\sum_{b\in \pi(e)}|\sqrt{a(b)}\nabla p_{t-s}(b,0)|^2w'(e)^{-1}\sum_{b'\in \pi'(e)}|\sqrt{a^{e}(b')}\nabla \bar u_s(a^{e},b')|^2 \biggr)^{p}\rangle^{{ \frac{1}{2p}}}\, ds\\
&\simeq\int_0^t \langle\biggl(\sum_b |\sqrt{a(b)}\nabla p_{t-s}(b,0)|^2\sum_{e\in \pi^{-1}(b)}w(e)^{-1}w'(e)^{-1}\sum_{b'\in \pi'(e)}|\sqrt{a^{e}(b')}\nabla \bar u_s(a^{e},b')|^2 \biggr)^{p}\rangle^{{ \frac{1}{2p}}}\, ds,
\end{align*} 
where $\pi'(e)$ and $w'(e)$ are as $\pi$ and $w$ introduced in Lemma \ref{w-mod} and related to the environment given by $a^e$. 
After smuggling the weight $\omega_\alpha^2(t-s,\underline b)$ in the sum over $b$, we use H\"older's inequality with exponents $q$ and $p$ to estimate
\begin{align*}
&\int_0^t \langle\biggl(\sum_e |\nabla p_{t-s}(e,0)|^2| \nabla \bar u_s(a^{e},e)|^2 \biggr)^{p}\rangle^{{ \frac{1}{2p}}}\, ds \\
&\lesssim \int_0^t \langle\biggl(\sum_b \omega_{\alpha}^{2q}(t-s,\underline b)|\sqrt{{a}(b)}\nabla p_{t-s}(b,0)|^{2q} \biggl)^{p-1} \\
&\hspace{0.1cm} \times \biggl(\sum_b \omega_{\alpha}^{-2p}(t-s,\underline b)\bigl(\sum_{e\in \pi^{-1}(b)}w(e)^{-1}w'(e)^{-1}\sum_{b'\in \pi'(z)}|\sqrt{a^{e}(b')}\nabla \bar u_s(a^{e},b')|^2 \bigr)^p\biggr)\rangle^{{ \frac{1}{2p}}}\, ds.
\end{align*} 
We now appeal to Lemma \ref{LG} and the embedding $\ell^2 \subset \ell^{r}$ for $r \geq 2$ to infer from the previous inequality that
\begin{align}\label{T23}
&\int_0^t \langle\biggl(\sum_e |\nabla p_{t-s}(e,0)|^2| \nabla \bar u_s(a^{e},e)|^2 \biggr)^{p}\rangle^{{ \frac{1}{2p}}}\, ds\notag \\
&\lesssim \int_0^t (t-s)^{-(\frac d { 4} +{ \frac 1 2})(1-\frac 1 p)}\biggl( \sum_b \omega_{\alpha}^{-2p}(t-s ,\underline b) \notag\\
&\hspace{0.5cm}\times \langle\mathcal{Y}_{t-s}^{p-1}\bigl(\sum_{e\in \pi^{-1}(b)}w(e)^{-1}w'(e)^{-1}\sum_{b'\in \pi'(z)}|\sqrt{a^{e}(b')}\nabla \bar u_s(a^{e},b')|^2 \bigr)^p\rangle \biggr)^{{ \frac{1}{2p}}}\, ds.
\end{align} 
We claim that
\begin{align}\label{T19}
\langle \mathcal{Y}_{t-s}^{p-1}\bigl(\sum_{e\in \pi^{-1}(b)}w(e)^{-1}w'(e)^{-1}\sum_{b'\in \pi'(e)}|\sqrt{a^e(b')}&\nabla \bar u_s(a^{e},b')|^2 \bigr)^p\rangle\notag\\
&\lesssim C(\delta) \langle | Du_s \cdot a Du_s|^{p(1+\delta)} \rangle^{\frac{1}{(1+\delta)}}.
\end{align}
We note that, once that we have (\ref{T19}), by inserting it in (\ref{T23}) we get
\begin{align*}
&\int_0^t \langle\biggl(\sum_e |\nabla p_{t-s}(e,0)|^2| \nabla \bar u_s(a^{e},e)|^2 \biggr)^{p}\rangle^{{ \frac{1}{2p}}}\, ds\notag \\
&\lesssim C(\delta)\int_0^t (t-s)^{-(\frac d 2 +1)(1-\frac 1 p)}\biggl( \sum_b \omega_{\alpha}^{-2p}(t-s,\underline b) \langle | Du_s \cdot a Du_s|^{p(1+\delta)} \rangle^{\frac{1}{(1+\delta)}} \biggr)^{{ \frac{1}{2p}}}\, ds\\
&\stackrel{\alpha > \frac d 2} {\lesssim}C(\delta)\int_0^t (t-s)^{-(\frac d { 4} +{ \frac 1 2})(1-\frac 1 p)+ \frac d p}\langle | Du_s \cdot a Du_s|^{p(1+\delta)} \rangle^{\frac{1}{{ 2}p(1+\delta)}} \, ds
\end{align*}
and thus (\ref{T17}).

\medskip

We now prove (\ref{T19}): It holds
\begin{align*}
\langle\mathcal{Y}_{t-s}^{p-1}&\bigl(\sum_{e\in \pi^{-1}(b)}w(e)^{-1}w'(e)^{-1}\sum_{b'\in \pi'(e)}|\sqrt{a^{e}(b')}\nabla \bar u_s(a^{e},b')|^2 \bigr)^p\rangle\\
\leq &\langle\mathcal{Y}_{t-s}^{p-1}|\pi^{-1}(b)|^{p-1} \sum_{e\in \pi^{-1}(b)}w(e)^{-p}w'(e)^{-p}\bigl(\sum_{b'\in \pi'(e)}|\sqrt{a^{e}(b')}\nabla \bar u_s(a^{e},b')|^2 \bigr)^p\rangle\\
=& \langle \sum_{e\in \pi^{-1}(b)}\mathcal{Y}_{t-s}^{p-1}|\pi^{-1}(b)|^{p-1}w(e)^{-p}w'(e)^{-p}\bigl(\sum_{b'\in \pi'(e)}|\sqrt{a^{e}(b')}\nabla \bar u_s(a^{e},b')|^2 \bigr)^p\rangle.
\end{align*} 
H\"older's inequality with exponents $(1+\delta, \frac{1+\delta}{\delta})$ first in $e$ and then in $\langle \cdot \rangle$ yields
\begin{align}\label{T20}
\langle\mathcal{Y}_{t-s}^{p-1}\bigl(&\sum_{e\in \pi^{-1}(b)}w(e)^{-1}w'(e)^{-1}\sum_{b'\in \pi'(e)}|\sqrt{a^{e}(b')}\nabla \bar u_s(a^{e},b')|^2 \bigr)^p\rangle\notag\\
&\leq \langle \sum_{e\in \pi^{-1}(b)}\mathcal{Y}_{t-s}^{(p-1)\frac{1+\delta}{\delta}}|\pi^{-1}(b)|^{(p-1)\frac{1+\delta}{\delta}}w(e)^{-p\frac{1+\delta}{\delta}}w'(e)^{-p\frac{1+\delta}{\delta}}\rangle^{\frac{\delta}{1+\delta}}\notag\\
&\hspace{3cm}\times \langle \sum_{e\in \pi^{-1}(b)}\bigl(\sum_{b'\in \pi'(e)}|\sqrt{a^{e}(b')}\nabla \bar u_s(a^{e},b')|^2 \bigr)^{p(1+\delta)}\rangle^{\frac{1}{1+\delta}}
\end{align}
As the terms $\pi'(e)$ and $\sqrt{a^{e}(b')}\nabla \bar u_s(a^{e},b')$ are stationary respectively by Lemma \ref{w-mod} and by definition of stationary extension, we have that also $\sum_{b'\in \pi'(e)}|\sqrt{a^{e}(b')}\nabla \bar u_s(a^{e},b')|^2$ is stationary. Therefore, appealing to Lemma \ref{SUM} for every $\delta >0$ it holds
\begin{align*}
\langle\sum_{e\in \pi^{-1}(b)}&\bigl(\sum_{b'\in \pi'(e)}|\sqrt{a^{e}(b')}\nabla \bar u_s(a^{e},b')|^2 \bigr)^{p(1+\delta)}\rangle^{\frac{1}{1+\delta}}\notag\\
&\stackrel{(\ref{SUM2})}{\lesssim} C(\delta)\sum_{i=1}^{d}\langle \bigl(\sum_{b'\in \pi'(e_i)}|\sqrt{a^{e_i}(b')}\nabla \bar u_s(a^{e_i},b')|^2 \bigr)^{p(1+\delta)^2}\rangle^{\frac{1}{(1+\delta)^2}}\notag\\
&\stackrel{(\ref{SUM1})}{\lesssim}  C(\delta)\sum_{i=1}^{d}\langle |\sqrt{a^{e_i}}D u_s(a^{e_i})|^{2p(1+\delta)^3}\rangle^{\frac{1}{(1+\delta)^3}}.
\end{align*}
Since $a^{e}$ and $a$ have by definition the same law, we conclude that
\begin{align}\label{T21}
\langle\sum_{e\in \pi^{-1}(b)}\bigl(\sum_{b'\in \pi'(e)}|\sqrt{a^{e}(b')}\nabla \bar u_s(a^{e},b')|^2 &\bigr)^{p(1+\delta)}\rangle^{\frac{1}{1+\delta}}\notag\\
&\lesssim  C(\delta)\langle |\sqrt{a}D u_s(a)|^{2p(1+\delta)^3}\rangle^{\frac{1}{(1+\delta)^3}}.
\end{align}
We now turn to the other averaged term in (\ref{T20}): Reasoning as above, thanks to (iii) of Lemma \ref{w-mod} and (\ref{LG2b}) of Lemma \ref{LG}, we may apply (\ref{SUM1}) of Lemma \ref{SUM} and get that for every $\delta >0$ the last term on the r.h.s. of \eqref{T20} satisfies
\begin{align}\label{T22}
\langle \sum_{e\in \pi^{-1}(b)}\mathcal{Y}_{t-s}^{(p-1)\frac{1+\delta}{\delta}}&|\pi^{-1}(b)|^{(p-1)\frac{1+\delta}{\delta}}w(e)^{-p\frac{1+\delta}{\delta}}w'(e)^{-p\frac{1+\delta}{\delta}}\rangle\notag\\
\lesssim &\sum_{i=1}^d\langle \bigl(\mathcal{Y}_{t-s}^{(p-1)\frac{1+\delta}{\delta}}|\pi^{-1}(e_i)|^{(p-1)\frac{1+\delta}{\delta}}w(e_i)^{-p\frac{1+\delta}{\delta}}w'(e_i)^{-p\frac{1+\delta}{\delta}}\bigr)^{1+\delta}\rangle^{\frac{1}{1+\delta}}\notag\\
&\stackrel{(\ref{W2})-(\ref{W3})-(\ref{LG2c})}{\lesssim} C(\delta).
\end{align}
Inequalities (\ref{T21}) and (\ref{T22}) in (\ref{T20}) yield the bound (\ref{T19}), once that we relabel $\delta \simeq \delta^3 + 2\delta^2 + 2\delta$.

\bigskip

We are now ready to conclude the proof of Theorem \ref{t.main}, part (a): By (\ref{T16}) and (\ref{T17}), we get from (\ref{RF})
\begin{align*}
\langle |u_t|^{2p} \rangle^{{ \frac{1}{2p}}} \lesssim & \ C(\delta) \biggl( N||g||_{L^\infty(\Omega)} t^{-\frac{d}{ { 4}}}\\
 &+ \int_0^t (t-s)^{-(\frac {d} {{ 4}} +{ \frac 1 2})(1 -\frac 1 p)+ \frac{d}{2p}} \langle |D u_s \cdot a D u_s|^{p(1+\delta)} \rangle^{\frac{1}{{ 2}p(1+\delta)}} \, ds \biggr).
\end{align*}

By our assumption that $g\in L^\infty(\Omega)$ and the maximum principle, we have
\begin{equation*}  
\langle |D u_s \cdot a D u_s|^{p(1+\delta)} \rangle^{\frac{1}{{ 2}p(1+\delta)}} \lesssim ||g||_{L^\infty(\Omega)}^{\frac{\delta}{\delta + 1}} \langle |D u_s \cdot a D u_s|^{p} \rangle^{\frac{1}{{ 2}p(1+\delta)}},
\end{equation*}
and therefore
\begin{align*}  
\langle |u_t|^{2p} \rangle^{{ \frac{1}{2p}}} &\lesssim   C(\delta)\biggl( N ||g||_{L^\infty(\Omega)} t^{-\frac{ d} { 4}}\\
& + \|g\|_{L^\infty(\Omega)}^{\frac{\delta}{\delta + 1}}\int_0^t (t-s)^{-(\frac{d}{ { 4}} +{ \frac 1 2})(1 -\frac 1 p)+ \frac {d}{{ 2}p}} \langle |D u_s \cdot a D u_s|^{p} \rangle^{\frac{1}{{ 2}p(1+\delta)}} \, ds \biggr).
\end{align*}
For $p$ sufficiently large, we now use Lemma~\ref{ODE} with { $\gamma_0:= \frac d 4$, $\gamma_1:= (\frac{d}{ { 4}} +{ \frac 1 2})(1 -\frac 1 p)- \frac {d}{{ 2}p} \geq 1$}, any $\beta:= \frac{1}{1+\delta} > \frac{\gamma }{\gamma+ \frac {1}{2p}}$ and for the functions
\begin{align*}  
a(t) & := ( N \|g\|_{L^\infty(\Omega)} )^{-1} \langle |u_t|^{2p} \rangle^{{ \frac{1}{2p}}}, \\
b(t) & :=  ( N^{1+\delta} \|g\|_{L^\infty(\Omega)} )^{-1} \langle |D u_s \cdot a D u_s|^{p} \rangle^{{ \frac{1}{2p}}}.
\end{align*}
We remark that the relation~\eqref{ODE2} is satisfied since by Lemma~\ref{Mono} and $N \geq 1$ we have
\begin{align*}
b(t)^{{ 2}p} &= ( N^{1+\delta} \|g\|_{L^\infty(\Omega)} )^{-{{ 2}p}} \langle |D u_s \cdot a D u_s|^{p} \rangle \lesssim (N^{1+\delta} \|g\|_{L^\infty(\Omega)} )^{-{{ 2}p}}\biggl( - \frac{d}{dt}\langle |u_t|^{2p} \rangle \biggr)\\
&\lesssim ( N \|g\|_{L^\infty(\Omega)} )^{-{{ 2}p}}\biggl( - \frac{d}{dt}\langle |u_t|^{2p} \rangle \biggr) \simeq \biggl( - \frac{d}{dt} a^{{{ 2}p}}(t)\biggr).
\end{align*}
The embedding 
\begin{align}\label{emb}
\langle |u_t|^q \rangle^{\frac 1 q} \leq C(p,q) \langle |u_t |^p \rangle^{\frac 1 p}
\end{align}
allows to conclude the proof of Theorem \ref{Teo1}, part (a).

\bigskip

To prove also Theorem \ref{t.main}, part (b) we observe that thanks to the assumption $g = D^*f$ (and thus $\bar g = \nabla^*\bar f$), after an integration by parts, we may reformulate the estimate (\ref{RF}) as
\begin{align}\label{RF2}
 \langle |u_t|^{2p}\rangle^{{ \frac{1}{2p}}} \lesssim \langle \biggl(\sum_{e} \bigl( \sum_{b} \nabla p_t(b,0)\partial_{e}\overline{f}&(b) \bigr)^{2}\biggr)^{p}\rangle^{{ \frac{1}{2p}}}\notag\\
 + &\int_0^t \langle\biggl(\sum_e |\nabla p_{t-s}(e,0)|^2| \nabla \bar u_s(a^{e},e)|^2 \biggr)^{p}\rangle^{{ \frac{1}{2p}}}.
\end{align}
We control the second term on the r.h.s. again by (\ref{T17}). For the first term, we now argue that 
\begin{align}\label{T18}
\langle \biggl(\sum_{e} \bigl( \sum_{b} \nabla p_t(b,0)\partial_{e}&\overline{f}(b) \bigr)^{2}\biggr)^{p}\rangle^{\frac{1}{p}}\lesssim N^{ 2} ||f||^{2}_{ L^\infty(\Omega)}t^{-(\frac d 2 + 1)(1 - \frac 1 p)+ \frac d p},
\end{align}
with $N$ the number of sites on which $f$ depends.

\medskip

Before proving (\ref{T18}), we show how to conclude the argument: Inserting  inequalities (\ref{T18}),(\ref{T16}) in (\ref{RF2}) and using the maximum principle and that $f \in L^\infty(\Omega)$, we get
\begin{multline*}  
 \langle |u_t|^{2p}\rangle^{{ \frac{1}{2p}}} \lesssim C(\delta)\biggl( N ||f||_{ L^\infty(\Omega)} \, t^{-(\frac{d}{ 4} +{ \frac 1 2})(1-\frac 1 p)  + \frac{d}{2p}} \\
 + \|f\|_{L^\infty(\Omega)}^{\frac{\delta}{1+\delta}} \int_0^t (t-s)^{-( \frac{d}{{ 4}} +{ \frac 1 2 })(1-\frac 1 p)  + \frac{d}{2p}} \langle |D u_s \cdot a D u_s|^{p} \rangle^{\frac{1}{{ 2}p(1+\delta)}} \, ds \biggr).
\end{multline*}
As for the proof of part (a), we multiply the previous inequality by $\bigl( C(\delta) N ||f||_{L^\infty(\Omega)} \bigr)^{-1}$ and use Lemma \ref{Mono} and Lemma \ref{ODE} applied to
\begin{align*}  
a(t) & := \bigl( N ||f||_{L^\infty(\Omega)} \bigr)^{-1} \langle |u_t|^{2p} \rangle^{{ \frac{1}{2p}}}, \\
b(t) & := \bigl( N^{1+\delta} ||f||_{L^\infty(\Omega)} \bigr)^{-1} \langle |D u_s \cdot a D u_s|^{p} \rangle^{{ \frac{1}{2p}}} ,
\end{align*}\
with { $\gamma_0=\gamma:= (\frac d 4 +\frac 1 2)(1-\frac 1 p) - \frac {d}{ 2p}$} and any $\beta:= \frac{1}{1+\delta} > \frac{\gamma }{\gamma+ \frac {1}{ 2p}}$ to obtain
\begin{align*}
 \langle |u_t|^{2p}\rangle^{\frac{1}{p}} \lesssim N^{ 2} ||f||^{2}_{ L^\infty(\Omega)} t^{-(\frac d 2 +1)(1-\frac 1 p) + \frac d p} \lesssim N^{ 2}||f||^{2}_{ L^\infty(\Omega)}  t^{-(\frac d 2 +1) + \frac{d}{2p}}.
\end{align*}
Therefore, for every $q \in [1, +\infty)$ fixed, let us consider $\varepsilon > 0$; we choose $q \leq p < +\infty$ such that $\frac{d}{2p} \leq \varepsilon$ and use the previous estimate together with the embedding \eqref{emb} to infer 
\begin{align*}
\langle |u_t |^{2q} \rangle^{\frac 1 q} \lesssim C(\varepsilon) N^{ 2} ||f||^{2}_{ L^\infty(\Omega)} t^{-(\frac d 2 +1) + \frac{d}{2p}} \lesssim C(\varepsilon) N^{ 2} ||f||^{2}_{ L^\infty(\Omega)} t^{-(\frac d 2 +1) + \varepsilon},
\end{align*}
i.e. bound (\ref{Teo2}).

\medskip

It thus only remains to show (\ref{T18}): As in (\ref{T15}), we bound the l.h.s of (\ref{T18}) by
\begin{align}\label{T24}
\langle \biggl(\sum_{e} \bigl( \sum_{b} \nabla p_t(b,0)\partial_{e}&\overline{f}(z) \bigr)^{2}\biggr)^{p}\rangle^{{ \frac{1}{2p}}}\notag\\
\leq &\sum_{i=1, ...,d}\sum_{z} \langle \biggl(\sum_{b} |\nabla  p_t(b,0)|^2|\overline{\partial_{\{z,z+e_i\}}f}(b)|^2 \biggr)^{p}\rangle^{{ \frac{1}{2p}}}.
\end{align}
For any fixed $i=1,...,d$, inequality (\ref{W1b}) of Lemma \ref{w-mod} yields
\begin{align*}
\sum_{z} \langle \biggl(\sum_{b}|\nabla & p_t(b,0)|^2|\overline{\partial_{\{z,z+e_i\}}f}(b)|^2 \biggr)^{p}\rangle^{{ \frac{1}{2p}}}\\
&\leq \sum_{z} \langle \biggl(\sum_{b} w(b)^{-1}\sum_{b'\in \pi(b)}|\sqrt{a(b')}\nabla  p_t(b',0)|^2|\overline{\partial_{\{z,z+e_i\}}f}(b)|^2 \biggr)^{p}\rangle^{{ \frac{1}{2p}}}\\
&= \sum_{z} \langle \biggl(\sum_{b'} |\sqrt{a(b')}\nabla  p_t(b',0)|^2\sum_{b\in \pi^{-1}(b')}w(b)^{-2}|\overline{\partial_{\{z,z+e_i\}}f}(b)|^2 \biggr)^{p}\rangle^{{ \frac{1}{2p}}}
\end{align*}
After smuggling the weight $\omega_\alpha^2(t, '\underline b)$ into the sum over $b'$, H\"older's inequality with exponents $p$ and $q$ implies
\begin{align*}
\sum_{z} \langle \biggl(\sum_{b}|\nabla & p_t(b,0)|^2|\overline{\partial_{\{z,z+e_i\}}f}(b)|^2 \biggr)^{p}\rangle^{{ \frac{1}{2p}}}\\
&\lesssim \sum_{z} \langle \biggl(\sum_{b'}\omega_\alpha^{2q}(t, \underline b') |\sqrt{a(b')}\nabla  p_t(b',0)|^{2q} \biggr)^{p-1}\\
&\hspace{2cm} \times\biggl(\sum_{b'} \omega_\alpha^{-2p}(t, \underline b')\bigl(\sum_{b\in \pi^{-1}(b')}w(b)^{-1}|\overline{\partial_{\{z,z+e_i\}}f}(b)|^2\bigr)^p \biggr)\rangle^{{ \frac{1}{2p}}},
\end{align*}
and thus by Lemma \ref{LG},
\begin{align}\label{T25}
\sum_{z} \langle \biggl(\sum_{b}&|\nabla  p_t(b,0)|^2|\overline{\partial_{\{z,z+e_i\}}f}(b)|^2 \biggr)^{p}\rangle^{{ \frac{1}{2p}}}\notag\\
&\lesssim t^{-(\frac {d}{ 4} + {  \frac 1 2})(1 - \frac 1 p)}\sum_{z} \biggl(\sum_{b'} \omega_\alpha^{-2p}(t, \underline b') \langle\bigl(\sum_{b\in \pi^{-1}(b')}w(b)^{-1}|\overline{\partial_{\{z,z+e_i\}}f}(b)|^2\bigr)^p\rangle \biggr)^{{ \frac{1}{2p}}}.
\end{align}
We now show that
\begin{align}\label{T26}
\langle\bigl(\sum_{b\in \pi^{-1}(b')}w(b)^{-1}|\overline{\partial_{\{z,z+e_i\}}f}(b)|^2\bigr)^p\rangle \lesssim ||\partial_{\{z,z+e_i \}}f||^{2p}_{ L^\infty(\Omega)}.
\end{align}
Once proven the previous inequality, estimate (\ref{T25}) turns into 
\begin{align*}
\sum_{z} \langle \biggl(\sum_{b}&|\nabla  p_t(b,0)|^2|\overline{\partial_{\{z,z+e_i\}}f}(b)|^2 \biggr)^{p}\rangle^{{ \frac{1}{2p}}}\notag\\
&\lesssim  t^{-(\frac {d}{ 4} + {  \frac 1 2})(1 - \frac 1 p)}\sum_{z} \biggl(\sum_{b'} \omega_\alpha^{-2p}(\underline b',t) \biggr)^{{ \frac{1}{2p}}} ||\partial_{\{z,z+e_i \}}f||_{ L^\infty(\Omega)}\\
&\lesssim t^{-(\frac {d}{ 4} + {  \frac 1 2})+ \frac {d}{2 p}} \sum_{z} ||\partial_{\{z,z+e_i \}}f||_{ L^\infty(\Omega)}\\
&\lesssim Nt^{--(\frac {d}{ 4} + {  \frac 1 2})+ \frac {d}{2 p}} ||f||_{ L^\infty(\Omega)},
\end{align*}
where for the last inequality we observe that $\partial_{z, z+e_i}f=0$ if $\{ z, z+e_i \} \notin$ supp$(f)$.
The bound (\ref{T18}) follows by (\ref{T24}).

\medskip

To obtain (\ref{T26}), we reason similarly to (\ref{T21}) and (\ref{T22}): We write
\begin{align*}
\langle \bigl(\sum_{b\in \pi^{-1}(b')}&w(b)^{-1}|\overline{\partial_{\{z,z+e_i\}}f}(b)|^2\bigr)^p\rangle\\
&\stackrel{(\ref{SUM1})}{\lesssim}\sum_{j=1}^d C(\delta) \langle w(e_j)^{-p(1+\delta)}  |\overline{\partial_{\{z,z+e_i\}}f}(e_j)|^{2p(1+\delta)}| \rangle^{\frac{1}{1+\delta}}\\
&\stackrel{\delta=1}{\lesssim}\sum_{j=1}^d  \langle w(e_j)^{-2p}  |\overline{\partial_{\{z,z+e_i\}}f}(e_j)|^{4p}| \rangle^{\frac{1}{2}}.
\end{align*}
Appealing  to the assumption $f\in L^\infty(\Omega)$ and to the bound (\ref{W2}) of Lemma \ref{w-mod}, we infer (\ref{T26}).
\end{proof}
\section{Appendix}
\begin{proof}[Proof of Lemma \ref{MaxF}] 
We start observing that it is enough to show an analogue of the maximal function estimate (\cite{AK}, Theorem 3.2), namely that for a constant $C=C(d, \alpha)< +\infty$ it holds
\begin{align}\label{MF1}
\P( | \sup_{t > 0} \, {t}^{-\frac d 2} \sum_{z} \omega_{\alpha}^{-2}(z, t) \mathcal{Z}(\tau_z a) | > \lambda ) \leq C \frac{1}{\lambda}\langle |\mathcal{Z}| \rangle.
\end{align}
This indeed, combined with the Marcinkiewicz Interpolation Theorem and the fact that the map 
\begin{align*}
\mathcal{Z} \rightarrow \sup_{t > 0}  \biggl( {t}^{-\frac d 2} \sum_{z} \omega_{\alpha}^{-2}(z, t) \mathcal{Z}(\tau_z a) \biggr)
\end{align*}
is bounded from $L^\infty$ to $L^\infty$, yields for every $p \in (1, +\infty]$
\begin{align*}
\langle | \sup_{t > 0}\, \biggl( {t}^{-\frac d 2} \sum_{z} \omega_{\alpha}^{-2}(z, t) \mathcal{Z}(\tau_z a) \biggr)|^p \rangle \lesssim \langle |\mathcal{Z}|^p\rangle
\end{align*}
and thus also (\ref{MFa}) by assumption (\ref{MFb}). Above and in the rest of the proof $\lesssim$ stands for $\leq C(d, \alpha, p)$.

\medskip

We now give the argument for (\ref{MF1}): We observe that it holds
\begin{align*}
\P( &\, \sup_{t > 0} \,\biggl( {t}^{-\frac d 2} \sum_{z} \omega_{\alpha}^{-2}(z, t) \mathcal{Z}(\tau_z a)\biggr) > \lambda )\\
&=\P( \sup_{t > 0}\biggl( {t}^{-\frac d 2} \sum_{n=1}^\infty \sum_{(n-1)t^{\frac 1 2} \leq |z| < nt^{\frac 1 2}} \omega_{\alpha}^{-2}(z, t) \mathcal{Z}(\tau_z a) \biggr)  > \lambda )\\
&=\P(  \exists \  t \in (0, +\infty) \, : \, {t}^{-\frac d 2}\sum_{n=1}^{\infty}\sum_{(n-1)t^{\frac 1 2}\leq |z| < nt^{\frac 1 2}} \omega_{\alpha}^{-2}(z, t) \mathcal{Z}(\tau_z a) > \lambda )\\
&=\P(  \exists \  t\in (0, +\infty) , n\in \N \, : \, {t}^{-\frac d 2}\sum_{(n-1)t^{\frac 1 2}\leq |z| < nt^{\frac 1 2}} \omega_{\alpha}^{-2}(z, t) \mathcal{Z}(\tau_z a) \gtrsim \frac{\lambda}{n^2}\,).
\end{align*}
Since $\mathcal{Z} \geq 0$ and
$$
 \sup_{(n-1)t^{\frac 1 2}\leq |z| < nt^{\frac 1 2}}\omega_{\alpha}^{-2}(z,t) \lesssim n^{-2\alpha},
$$ 
we get
\begin{align*}
\P( &\, \sup_{t > 0}\,\biggl( {t}^{-\frac d 2} \sum_{z} \omega_{\alpha}^{-2}(z, t) \mathcal{Z}(\tau_z a)\biggr) > \lambda )\\
&\leq \P(  \exists \  t\in (0, +\infty) , n\in \N \, : \, n^{-d}{t}^{-\frac d 2}\sum_{(n-1)t^{\frac 1 2}\leq |z| < nt^{\frac 1 2}}\mathcal{Z}(\tau_z a) \gtrsim n^{2\alpha - d-2}\lambda\,)\\
&\leq \P( \exists \  t\in (0, +\infty) , n\in \N \, : \, n^{-d}{t}^{-\frac d 2}\sum_{|z| < nt^{\frac 1 2}}\mathcal{Z}(\tau_z a) \gtrsim n^{2\alpha - d-2}\lambda\,)\\
&\leq \sum_{m=1}^\infty\P( \exists \  t\in (0, +\infty) , n\in \N \, : \, n^{-d}{t}^{-\frac d 2}\sum_{|z| < nt^{\frac 1 2}}\mathcal{Z}(\tau_z a) \gtrsim m^{2\alpha - d-2}\lambda\,)\\
&\leq \sum_{m=1}^\infty\P( \sup_{R > 0} \biggl( R^{-d}\sum_{|z| < R}\mathcal{Z}(\tau_z a)\biggr) \gtrsim m^{2\alpha - d-2}\lambda\,).
\end{align*}
We may now use in this last line the standard maximal function estimate (\cite{AK}, Theorem 3.2)
\begin{align*}
\P(\sup_{R>0} \biggl( R^{-d}\sum_{|z|< R}\mathcal{Z}(\tau_z a)\biggr) > \delta) \leq \frac{\langle |\mathcal{Z}|\rangle}{\delta}
\end{align*}
to conclude 
\begin{align*}
\P(  \sup_{t > 0}& \,\biggl( {t}^{-\frac d 2} \sum_{z} \omega_{\alpha}^{-2}(z, t) \mathcal{Z}(\tau_z a)\biggr)> \lambda )\lesssim \sum_{m=1}^\infty  \frac{\langle |\mathcal{Z}|\rangle}{m^{2\alpha-2-d}\lambda} \stackrel{\alpha > \frac d 2 + 1}{\lesssim} \frac{\langle |\mathcal{Z}|\rangle}{\lambda}.
\end{align*}
\end{proof}
\bibliographystyle{abbrv}
\bibliography{Quantitative_Hom_Deg_Environments_rev.bbl}
\end{document}